\documentclass[11pt,leqno]{article}
\usepackage{amsmath,amsfonts,amscd,amsthm}
\begin{document}

\newcommand{\se}{\setcounter{equation}{0}}
\def\theequation{\thesection.\arabic{equation}}

\newtheorem{theorem}{Theorem}[section]
\newtheorem{cdf}{Corollary}[section]
\newtheorem{lemma}{Lemma}[section]
\newtheorem{remark}{Remark}[section]

\newcommand{\R}{\mathbb{R}}
\newcommand{\e}{{\bf e}}
\newcommand{\E}{{\bf E}}
\newcommand{\f}{{\bf f}}
\newcommand{\he}{\hat{\e}}
\newcommand{\hf}{\hat{\f}}
\newcommand{\te}{\tilde{\e}}

\newcommand{\bu}{{\bf u}}
\newcommand{\hbu}{\hat{\bu}}
\newcommand{\bv}{{\bf v}}
\newcommand{\bw}{{\bf w}}
\newcommand{\by}{{\bf y}}
\newcommand{\bz}{{\bf z}}
\newcommand{\bH}{{\bf H}}
\newcommand{\bJ}{{\bf J}}
\newcommand{\bL}{{\bf L}}

\newcommand{\bchi}{\mbox{\boldmath $\chi$}}
\newcommand{\bta}{\mbox{\boldmath $\eta$}}
\newcommand{\tbta}{{\tilde{\bta}}}
\newcommand{\bphi}{\mbox{\boldmath $\phi$}}
\newcommand{\bth}{\mbox{\boldmath $\theta$}}
\newcommand{\bxi}{\mbox{\boldmath $\xi$}}
\newcommand{\bzt}{\mbox{\boldmath $\zeta$}}

\newcommand{\hta}{\hat{\bta}}
\newcommand{\hth}{\hat{\bth}}
\newcommand{\hxi}{\hat{\bxi}}
\newcommand{\hzt}{\hat{\bzt}}

\newcommand{\td}{\tilde{\Delta}}

\newcommand{\bbu}{\bar{\bu}}
\newcommand{\bby}{\bar{\by}}
\newcommand{\bbz}{\bar{\bz}}

\def\cydot{\leavevmode\raise.4ex\hbox{.}}

\title
{ \large\bf A Two-level Finite Element Method for Viscoelastic Fluid Flow: \\
Non-smooth Initial Data}
\author{Deepjyoti Goswami\footnote{
Department of Mathematics, Universidade Federal do Paran\'a, Centro Polit\'ecnico, Curitiba,
Cx.P: 19081, CEP: 81531-990, PR, Brazil}}
\date{}
\maketitle

\begin{abstract}
In this article, we analyze a two-level finite element method for the equations of motion arising in the flow of $2D$ Oldroyd model with non-smooth initial data. It involves solving the non-linear problem on a coarse grid of mesh-size $H$ and solving a linearized problem on a fine grid of mesh-size $h,~h<<H$. The method gives optimal convergence rate for velocity in $H^1$-norm and for pressure in $L^2$-norm. The analysis takes in to account the loss of regularity of the solution of the Oldroyd model at initial time.
\end{abstract}

\vspace{0.30cm} 

\vspace{.2in}
\section{Introduction}
We consider a two-level semi-discrete Galerkin approximations to the 
following system of equations of motion arising in the Oldroyd fluids of order one (see \cite{Old}):
\begin{eqnarray}\label{om}
 ~~\frac {\partial \bu}{\partial t}+\bu\cdot\nabla\bu-\mu\Delta\bu-\int_0^t \beta
 (t-\tau)\Delta\bu (x,\tau)\,d\tau+\nabla p=\f(x,t),~~x\in \Omega,~t>0
\end{eqnarray}
with incompressibility condition
\begin{eqnarray}\label{ic}
 \nabla \cdot \bu=0,~~~x\in \Omega,~t>0,
\end{eqnarray}
and initial and boundary conditions
\begin{eqnarray}\label{ibc}
 \bu(x,0)=\bu_0~~\mbox {in}~\Omega,~~~\bu=0,~~\mbox {on}~\partial\Omega,~t\ge 0.
\end{eqnarray}
Here, $\Omega$ is a bounded domain in $\R^2$ with boundary $\partial \Omega$,
$\mu = 2 \kappa\lambda^{-1}>0$ and the kernel $\beta (t) = \gamma \exp (-\delta t),$
where $\gamma= 2\lambda^{-1}(\nu-\kappa \lambda^{- 1})>0$ and $\delta =\lambda^{-1}>0$.
$\bu$ and $p$ are the velocity field and pressure, respectively. Further, the forcing term $\f$ and the initial velocity $\bu_0$ are given functions in their respective domains of definition. For more details, we refer to \cite{GP11} and references cited therein.

There is a considerable amount of literature devoted to Oldroyd model by Oskolkov, 
Kotsiolis, Karzeeva, Sobolevskii etc, see \cite{AS89,EO92,KrKO91,KOS92,Os89} 
and recently by Lin, He {\it et al.} \cite{HLST,HLSST,WHS10} and by Pani {\it et al.}
\cite{PY05, PYD06}. A brief introduction on the continuous and semi-discrete cases can be
found in \cite{GP11}. In fact, in \cite{GP11}, we have established {\it a priori} estimates
and regularity results for the solution pair $\{\bu,p\}$ of (\ref{om})-(\ref{ibc}) under
realistically assumed initial data and when $\f,\f_t\in L^{\infty}(\R_{+};\bL^2(\Omega))$.

Semi-discrete Galerkin finite element approximations and error analysis for this problem are carried out in \cite{HLSST, PY05, GP11}, while fully discrete cases can be found in \cite{GP12, PYD06, WHS10}. Other works on this problem, that is, (\ref{om})-(\ref{ibc}), can be attributed to He, Wang {\it et al.}, see \cite{WHF11, WHL12a, WSY12}.
But there is no work on two-level or similar methods for this problem, to the best of
author's knowledge and this is the first work in this direction.

Two-level or two-grid methods are well-established and efficient methods for solving non-linear partial differential equations. It involves solving the non-linear problem on a coarse grid (rather than on a fine grid, thereby computationally cost effective), and solving a linearized problem on a fine grid. In other words, in the first step, we discretize the non-linear PDE on a coarse mesh, of mesh-size $H$ and compute an approximate solution, say, $\bu_H$. Then, in the second step, we formulate a linearized problem, out of the original one, using $\bu_H$ and discretize it on a fine mesh, of mesh-size $h,~h<<H$, thereby, compute an approximate solution, say, $\bu^h$. With appropriate $h,H$, we obtain same order of convergence of the error $\bu-\bu^h$, as that of the error obtained by semi-discrete finite element Galerkin approximation on the find grid; but with far less computational cost, since, instead of solving a large non-linear system, we solve a small non-linear system and a large linear system.

In this article, we study the following two-level finite element approximation for the problem (\ref{om})-(\ref{ibc}): First, we compute a semi-discrete Galerkin finite element approximations $(\bu_H,p_H)$, over a coarse mesh of mesh-size $H$. Then, we use the approximation $\bu_H$ to compute a semi-discrete Galerkin finite element approximations $(\bu^h,p^h)$ of the following linearized Oldroyd problem:
\begin{align}\label{second}
\bu_t-\nu\Delta\bu+\bu_H\cdot\nabla\bu_H-\int_0^t \beta(t-s)\Delta\bu(s)~ds+\nabla p =\f,~~\nabla \cdot \bu =0,~~\mbox{in}~\Omega
\end{align}
over a fine mesh of mesh-size $h,~h<<H$.

The above algorithm is nothing new and in fact, this and similar algorithms have been studied 
on numerous occasions for both Stokes and Navier-Stokes problems. But to the best of our 
knowledge, no study has been done for our problem and with non-smooth initial data (even for 
Navier-Stokes), i.e., initial velocity is in $\bH_0^1$ and not in $\bH_0^1\cap\bH^2$ or in 
higher order Sobolev space. Our main aim in this work is to do error analysis of this 
two-level method, under non-smooth initial data.

Two-grid method was first introduced by Xu \cite{Xu94,Xu96} for semi-linear elliptic problems 
and by Layton {\it et al.} \cite{Lay93,LL95,LT98} for steady Navier-Stokes equations. It was 
carried out for time dependent Navier-Stokes by Girault and Lions \cite{GL01} for 
semi-discrete case. The method may vary depending on the linearize problem, to be solved in 
the second step; like, in the case of Navier-Stokes, one can chose a Stokes problem or an 
Oseen problem or a Newton step to solve on the fine mesh. Several works in this direction, 
involving both semi-discrete and fully discrete analysis, can be found in \cite{AGS09, AS08, He04, HL11, HMR04, LH10, LH10a, LHL12, Ol99} and references therein.

We would like to note here that, similar to Navier-Stokes (see \cite{HR82} for a discussion), the regularity of the solution of our problem breaks down as $t\to 0$. To assume otherwise, the initial data have to satisfy certain non-local compatibility conditions that are very difficult to verify and physically irrelevant. We have avoided these conditions in \cite{GP11}, while carrying out semi-discrete error analysis, thereby implying breakdown of regularity of the solution at initial time. Here also, we have worked out under similar circumstances.

Recently, two articles \cite{FGN12, FGN12a}, in the context of Navier-Stokes, have taken this "loss of regularity of the solution at initial time" into account. For example, both these articles assume no more than second-order spatial derivative of the velocity bounded in $\bL^2$, up to initial time $t=0$. But in our case, due to non-smooth initial data, we have $\|\bu(t)\|_2 \sim O(t^{-1/2})$, see \cite[Theorem $3.1$]{GP11}. In other words, no more than 
first-order spatial  derivative of the velocity is bounded in $\bL^2$, up to initial time. 

\noindent In this article, we present the following results:
\begin{align*}
\|\bu_h(t)-\bu^h(t)\|_1 + \|p_h(t)-p^h(t)\| \le Kt^{-1}H^2,
\end{align*}
where $(\bu^h,p^h)$ are the solutions of the two-level method on a grid of mesh-size $h$ and $(\bu_h,p_h)$ are the solution of the standard Galerkin finite element, which satisfy the following: (see \cite{GP11})
\begin{align*}
\|\bu(t)-\bu_h(t)\|_1 + \|p(t)-p_h(t)\| \le Kt^{-1/2}h.
\end{align*}
Therefore, away from $0$, this two-level method produces result similar to standard Galerkin finite element for $H=O(h^{1/2})$. This implies that we obtain similar result at a less computational cost, since this two-level method treats the non-linearity only in a coarse grid rather than a fine grid treatment as in Galerkin finite element.

Apart from the error analysis, we have also looked into the stability of the method, by working out the error analysis. We would like to note here that a proper justification can only be done with the analysis of full discretization and numerical implementation. This will be done in our next effort.

In the following sections, we will assume $C$ and $K$ to be generic positive constants, where $K$ generally depends on the given data, that is, $\bu_0,\f$. And for the sake of convenience, we would write $\bv(t)$ as $\bv$ when there arises no confusion.

The article is organized as follows. In section $2$, we introduce some notations and preliminaries. Section $3$ deals with Galerkin finite element approximations, whereas two-level method is discussed in Section $4$. Error analysis for two-level is carried out in Section $5$.

\section{Preliminaries}
\se

For our subsequent use, we denote by bold face letters the $\R^2$-valued
function space such as
\begin{align*}
 \bH_0^1 = [H_0^1(\Omega)]^2, ~~~ \bL^2 = [L^2(\Omega)]^2.
\end{align*}
Note that $\bH^1_0$ is equipped with a norm
$$ \|\nabla\bv\|= \big(\sum_{i,j=1}^{2}(\partial_j v_i, \partial_j
 v_i)\big)^{1/2}=\big(\sum_{i=1}^{2}(\nabla v_i, \nabla v_i)\big)^{1/2}. $$
Further, we introduce the following divergence free function space:
\begin{align*}
\bJ_1 &= \{\bphi\in\bH_0^1 : \nabla \cdot \bphi = 0\}.
\end{align*}
For any Banach space $X$, let $L^p(0, T; X)$ denote the space of measurable $X$
-valued functions $\bphi$ on  $ (0,T) $ such that
$$ \int_0^T \|\bphi (t)\|^p_X~dt <\infty~~~\mbox {if}~~1 \le p < \infty, $$
and for $p=\infty$
$$ {\displaystyle{ess \sup_{0<t<T}}} \|\bphi (t)\|_X <\infty~~~\mbox {if}~~p=\infty. $$
Through out this paper, we make the following assumptions:\\
(${\bf A1}$). For ${\bf g} \in \bL^2$, let the unique pair of solutions $\{\bv
\in\bJ_1, q \in L^2 /R\} $ for the steady state Stokes problem
\begin{align*}
 -\Delta\bv + \nabla q = {\bf g}, ~~
 \nabla \cdot\bv = 0~~~\mbox {in}~~~\Omega;~~~~\bv|_{\partial\Omega}=0,
\end{align*}
satisfy the following regularity result
$$  \| \bv \|_2 + \|q\|_{H^1/R} \le C\|{\bf g}\|. $$
\noindent
(${\bf A2}$). The initial velocity $\bu_0$ and the external force $\f$ satisfy for some
positive constant $M_0,$ and for $T$ with $0<T \leq \infty$
$$ \bu_0\in\bJ_1,~\f,\f_t,\f_{tt} \in L^{\infty} (0, T ;\bL^2)~~~\mbox{with} ~~~
\|\bu_0\|_1 \le M_0,~~{\displaystyle{\sup_{0<t<T} }}\big\{\|\f\|,\|\f_t\|,
 \|\f_{tt}\|\big\} \le M_0. $$

\noindent Before going into the details, let us introduce the weak formulation of
(\ref{om})-(\ref{ibc}). Find a pair of functions $\{\bu(t), p(t)\},~t>0,$ such that
\begin{align}\label{wfh}
(\bu_t, \bphi) +\mu (\nabla \bu, \nabla \bphi)+(\bu\cdot\nabla \bu, \bphi)
&+\int_0^t \beta(t-s) (\nabla \bu(s),\nabla \bphi)~ds
= ( p, \nabla \cdot \bphi) + (\f,\bphi)~~\forall\bphi\in\bH_0^1, \nonumber \\
(\nabla \cdot \bu, \chi) =& 0~~~\forall \chi \in L^2. 
\end{align}
Equivalently, find  $\bu(t) \in {\bf J}_1,~t>0 $ such that
\begin{equation}\label{wfj}
 (\bu_t, \bphi) +\mu (\nabla \bu, \nabla \bphi )+( \bu \cdot \nabla \bu, \bphi)
 +\int_0^t \beta(t-s) (\nabla\bu(s),\nabla \bphi)ds=(\f,\bphi),~\forall\bphi
 \in {\bf J}_1.
\end{equation}

For existence and uniqueness, and the regularity of the solution of the problem
(\ref{wfh}) or (\ref{wfj}), we refer to \cite{GP11}. In fact, in our subsequent section, we will use various estimates of $\bu$ and $p$, without explicit mention of them. These estimates can be found in Lemma $3.1$-Theorem $3.2$, \cite{GP11}. \\
We present below the positivity of the kernel $\beta$, which will be used in our subsequent analysis. The result is borrowed from \cite[Lemma 2.1]{GP11}.
\begin{lemma}\label{pp} 
For arbitrary $\alpha>0$, $t^*>0$ and $\phi\in L^2(0,t^*)$, the following positive
definite property holds
$$ \int_0^{t^*}{\left(\int_0^t{\exp{[-\alpha(t- s)]}\phi(s)}\,ds
   \right)\phi(t)}\,dt\ge 0. $$
\end{lemma}
\begin{lemma} [Gronwall's Lemma]
 Let $g,h,y$ be three locally integrable non-negative functions  on the time
 interval $[0,\infty)$ such that for all $t\ge 0$
 $$  y(t)+G(t)\le C+\int_0^t h(s)~ds+\int_0^t g(s)y(s)~ds, $$
 where $G(t)$ is a non-negative function on $[0,\infty)$ and  $C\ge 0$ is a
 constant. Then,
 $$ y(t)+G(t)\le{\Big(}C+\int_0^t h(s)~ds{\Big)}exp{\Big(}\int_0^t g(s)~ds{\Big)}. $$
\end{lemma}

\section{Galerkin Finite Element Method}
\se

From now on, we denote $h$, with $0<h<1$, to be a real positive spatial discretization
parameter, tending to zero. Let  $\bH_h$ and $L_h$ be two family of finite element spaces,
finite dimensional subspaces of $\bH_0^1 $ and $L^2/\R$, respectively,
approximating velocity vector and the pressure. Assume that the following
approximation properties are satisfied for the spaces $\bH_h$ and $L_h$: \\
${\bf (B1)}$ For each $\bw \in\bH_0^1 \cap \bH^2 $ and $ q \in
H^1/\R$ there exist approximations $i_h w \in \bH_h $ and $ j_h q \in
L_h $ such that
$$
 \|\bw-i_h\bw\|+ h \| \nabla (\bw-i_h \bw)\| \le Ch^2 \| \bw\|_2,
 ~~~~\| q - j_h q \| \le Ch \| q\|_1.
$$
Further, suppose that the following inverse hypothesis holds for $\bw_h\in\bH_h$:
\begin{align}\label{inv.hypo}
 \|\nabla \bw_h\| \leq  Ch^{-1} \|\bw_h\|.
\end{align}
For $\bv, \bw, \bphi \in \bH_0^1$, set
$$ a(\bv, \bphi) = (\nabla \bv, \nabla \bphi),~~~
b(\bv, \bw,\bphi)= \frac{1}{2} (\bv \cdot \nabla \bw , \bphi)
   - \frac{1}{2} (\bv \cdot \nabla \bphi, \bw). $$
Note that the operator $b(\cdot, \cdot, \cdot)$ preserves the antisymmetric property of
the original nonlinear term, that is,
$$ b(\bv_h, \bw_h, \bw_h) = 0 \;\;\; \forall \bv_h, \bw_h \in {\bH}_h. $$
The discrete analogue of the weak formulation (\ref{wfh}) now reads as: Find $\bu_h(t)
\in \bH_h$ and $p_h(t) \in L_h$ such that $ \bu_h(0)= \bu_{0h} $ and for $t>0$
\begin{eqnarray}\label{dwfh}
(\bu_{ht}, \bphi_h) +\mu a (\bu_h,\bphi_h) &+& b(\bu_h,\bu_h,\bphi_h)+ \int_0^t \beta(t-s) a(\bu_h(s),\bphi_h)~ds -(p_h, \nabla \cdot \bphi_h) =(\f, \bphi_h), \nonumber \\
&&(\nabla \cdot \bu_h, \chi_h) =0,
\end{eqnarray}
for $\bphi_h\in\bH_h,~\chi_h \in L_h$. Here $\bu_{0h} \in\bH_h $ is a suitable 
approximation of $\bu_0\in\bJ_1$. Or we can define
$$ {\bf J}_h = \{ v_h \in \bH_h : (\chi_h,\nabla\cdot v_h)=0
 ~~~\forall \chi_h \in L_h \}, $$
for an equivalent weak formulation: Find $\bu_h(t)\in {\bf J}_h $ such that $\bu_h(0) =
\bu_{0h} $ and for $t>0$
\begin{equation}\label{dwfj}
~~~~ (\bu_{ht},\bphi_h) +\mu a (\bu_h,\bphi_h)+ \int_0^t \beta(t-s) a(\bu_h(s), \bphi_h)~ds
= -b( \bu_h, \bu_h, \bphi_h)+(\f,\bphi_h)~~\forall \bphi_h \in \bJ_h.
\end{equation}
For the well-posedness of (\ref{dwfh}) or (\ref{dwfj}), we refer to \cite{GP11, PY05}.
Note that the pressure $p_h (t) \in L_h$ is unique up to a constant and unique in a
quotient space $L_h/N_h$, where
$$ N_h=\{q_h\in L_h: (q_h,\nabla\cdot\bphi_h)=0 \mbox{ for } \bphi_h\in\bH_h \}. $$
The norm of $L_h/N_h$ is given by
$$ \|q_h\|_{L^2/N_h}= \inf_{\chi_h\in N_h} \|q_h+\chi_h\|. $$

\noindent For continuous dependence of the discrete pressure $p_h (t) \in L_h/N_h$ on the
discrete velocity $u_h(t) \in {\bf J}_h$, we assume the following discrete
inf-sup (LBB) condition: \\
\noindent
${\bf (B2')}$  For every $q_h \in L_h$, there exists a non-trivial function
$\bphi_h \in \bH_h$
such that
$$ |(q_h, \nabla\cdot \bphi_h)| \ge C \|\nabla \bphi_h \|\| q_h\|_{L^2/N_h}, $$
where  the constant $C>0$ is independent of $h$. \\
Moreover, we also assume that the following approximation property holds true
for ${\bf J}_h $. \\
\noindent
${\bf (B2)}$ For every $\bw \in {\bf J}_1 \cap \bH^2, $ there exists an
approximation $r_h \bw \in {\bf J_h}$ such that
$$ \|\bw-r_h\bw\|+h \| \nabla (\bw - r_h \bw) \| \le Ch^2 \|\bw\|_2 . $$
The $L^2$ projection $P_h:\bL^2\mapsto \bJ_h$ satisfies the following properties
(see \cite{HR82}): for $\bphi\in \bJ_h$,
\begin{equation}\label{ph1}
 \|\bphi- P_h \bphi\|+ h \|\nabla P_h \bphi\| \leq C h\|\nabla \bphi\|,
\end{equation}
and for $\bphi \in \bJ_1 \cap \bH^2,$
\begin{equation}\label{ph2}
 \|\bphi-P_h\bphi\|+h\|\nabla(\bphi-P_h \bphi)\|\le C h^2\|\td\bphi\|.
\end{equation}
We now define the discrete  operator $\Delta_h: \bH_h \mapsto \bH_h$ through the
bilinear form $a (\cdot, \cdot)$ as
\begin{eqnarray}\label{do}
 a(\bv_h, \bphi_h) = (-\Delta_h\bv_h, \bphi)~~~~\forall \bv_h, \bphi_h\in\bH_h.
\end{eqnarray}
Define $\td_h = P_h(-\Delta_h) $. The restriction of $\td_h$ to $\bJ_h$ is
invertible and we denote the inverse by $\td_h^{-1}$. Since $-\td_h$ is self-adjoint and 
positive definite, we define {\it discrete Sobolev norms} on $\bJ_h$ as follows:
$$ \|\bv_h\|_r = \|(-\td_h)^{r/2}\bv_h\|,~~\bv_h\in\bJ_h,~r\in\R. $$
We note that, in particular, $\|\bv_h\|_0=\|\bv_h\|$ and $\|\bv_h\|_1=\|\nabla\bv_h\|$ for $\bv_h\in\bJ_h$, and $\|\cdot\|_2$ and $\|\td_h\cdot\|$ are equivalent on $\bJ_h$. For further detail, we refer to \cite{HR82, HR90}.

\noindent Below, we present some estimates of the non-linear term (see ($3.7$) from \cite{HR90}).
\begin{lemma}\label{nonlin}
Suppose conditions (${\bf A1}$), (${\bf B1}$) and (${\bf  B2}$) are satisfied. Then there 
exists a positive constant $K$ such that for $\bv,\bw,\bphi\in\bH_h$, the following holds:
\begin{equation}\label{nonlin1}
 |(\bv\cdot\nabla\bw,\bphi)| \le C \left\{
\begin{array}{l}
 \|\bv\|^{1/2}\|\nabla\bv\|^{1/2}\|\nabla\bw\|^{1/2}\|\Delta_h\bw\|^{1/2}
 \|\bphi\|, \\
 \|\bv\|^{1/2}\|\Delta_h\bv\|^{1/2}\|\nabla\bw\|\|\bphi\|, \\
 \|\bv\|^{1/2}\|\nabla\bv\|^{1/2}\|\nabla\bw\|\|\bphi\|^{1/2}
 \|\nabla\bphi\|^{1/2}, \\
 \|\bv\|\|\nabla\bw\|\|\bphi\|^{1/2}\|\Delta_h\bphi\|^{1/2}, \\
 \|\bv\|\|\nabla\bw\|^{1/2}\|\Delta_h\bw\|^{1/2}\|\bphi\|^{1/2}
 \|\nabla\bphi\|^{1/2}
\end{array}\right.
\end{equation}
\end{lemma}

\noindent Examples of subspaces $\bH_h$ and $L_h$ satisfying assumptions (${\bf B1}$) and (${\bf B2}'$) are abundant in literature (see \cite{GP11} and the references cited there).
We present below a lemma dealing with {\it a priori} and regularity estimates of $\bu_h$.
\begin{lemma}\label{est.uh}
Let the assumptions of Lemma \ref{nonlin} hold. Additionally, let $0<\alpha< \min \{\mu\lambda_1/2,\delta\}$
The semi-discrete Galerkin approximation $\bu_h$ of the velocity $\bu$ satisfies, for $t>0,$
\begin{eqnarray}
\|\bu_h(t)\|^2+e^{-2\alpha t}\int_0^t e^{2\alpha s}\|\bu_h(s)\|_1^2~ds \le K, \label{uh01} \\
\|\bu_h(t)\|_1^2+e^{-2\alpha t}\int_0^t e^{2\alpha s}\big\{\|\bu_h(s)\|_2^2+\|\bu_{h,s}(s)\|^2 \big\}~ds \le K, \label{uh02} \\
(\tau^*(t))^{1/2}\big\{\|\bu_h(t)\|_2+\|p_h(t)\|_{H^1/\R}\} \le K, \label{uh03}
\end{eqnarray}
where $\tau^*(t)= \min \{1,t\}$ and $K$ depends only on the given data. In fact, $K$ is independent of $h$.
\end{lemma}
\begin{remark}
Here $\lambda_1>0$ is the least eigenvalue of the Stokes operator (see \cite{GP11}).
\end{remark}

\begin{proof}
The estimates can be found in Lemma 4.2 of \cite{GP11}. The additional estimates of $\bu_{h,t}$ and of $p$ can be derived as in the proof of Theorem $3.1$ of \cite{GP11}.
\end{proof}
\begin{lemma}\label{est1.uh}
Under the assumptions of Lemma \ref{est.uh}, the semi-discrete Galerkin approximation $\bu_h$ of the velocity $\bu$ satisfies, for $t>0$ and for $r\in \{0,1\},~i\in \{1,2\},~r+i \le 2$,
\begin{equation}\label{1uh01}
(\tau^*(t))^{r+2i-1}\|D_t^i\bu_h(t)\|_r^2+e^{-2\alpha t}\int_0^t (\tau^*(s))^{r+2i-1}
e^{2\alpha s} \|D_s^i\bu_h(s)\|_{r+1}^2 ~ds \le K,
\end{equation}
where $D_t^i=\frac{\partial^i}{\partial t^i}$. And
\begin{eqnarray}
e^{-2\alpha t}\int_0^t (\tau^*(s))^2 e^{2\alpha s}\{\|\bu_{h,ss}(s)\|^2 +\|p_{h,s}(s)\|_{H^1/\R}^2\}~ds \le K, \label{1uh02} \\
(\tau^*(t))^{3/2}\{\|\bu_{h,t}\|_2+\|p_{h,t}\|_{H^1/\R}\} \le K. \label{1uh03}
\end{eqnarray}
Here $K$ depends only on the given data. In particular, $K$ is independent of $h$ and $t$.
\end{lemma}

\begin{proof}
The proofs of (\ref{1uh01}) for $i=1$ and of (\ref{1uh02}) follow in a similar fashion as that of Theorem $3.2$ from \cite{GP11}. For the remaining estimates, we need to differentiate (\ref{dwfj}) twice with respect to time and choose $\bphi_h=(\tau^*(t))^3 e^{2\alpha t} \bu_{h,tt}$ and again follow the ideas of the proof of Theorem $3.2$ from \cite{GP11} to complete the rest of the proof.
\end{proof}

\noindent The following semi-discrete error estimates are proved in \cite[Theorems $5.1$ and $6.1$]{GP11}.
\begin{theorem}\label{errest}
Let $\Omega$ be a convex polygon and let the conditions (${\bf A1}$)-(${\bf A2}$)
and (${\bf B1}$)-(${\bf B2}$) be satisfied. Further, let the discrete initial velocity $\bu_{0h}\in \bJ_h$ with $\bu_{0h}=P_h\bu_0,$ where $\bu_0\in \bJ_1.$ Then,
there exists a positive constant $K$, that depends only on the given data and the
domain $\Omega$, such that for $0<T<\infty $ with $t\in (0,T]$
$$ \|(\bu-\bu_h)(t)\|+h\|\nabla(\bu-\bu_h)(t)\|+h\|(p-p_h)(t)\|\le Ke^{Kt}h^2t^{-1/2}. $$
\end{theorem}

For the error analysis of the two-level method, we need a couple of additional estimates of 
the error due to Galerkin finite element, that is, of $(\bu-\bu_h)(t)$. In fact, we need estimates of the time derivative of the error. As in the case of parabolic problem, we need to differentiate the error equation, that is, the equation in $(\bu-\bu_h)(t)$, with respect to time and derive several estimates to conclude the result. It is technical and lengthy, but follow the proof techniques of Section $5$ from \cite{GP11}. For the sake of completeness, we try to sketch a proof here. We will be as brief as possible, conveying only the ideas, just to keep it clean and concise.

\begin{lemma}\label{err.et}
Under the assumptions of Theorem \ref{errest}, the following result holds:
\begin{equation}\label{err.et1}
e^{-2\alpha t}\int_0^t \sigma_1(s)\|(\bu_s-\bu_{h,s})(s)\|^2 ds \le K(t)h^4,
\end{equation}
where $\sigma_1(t)=(\tau^*(t))^2 e^{2\alpha t}$ and $K(t)$ is in fact, of the form $Ke^{Kt}$, but is independent of $h$. $\alpha$ is given by Lemma \ref{est.uh}.
\end{lemma}

\begin{proof}
We first recall the equation of $\E=\bu-\bu_h$ (see \cite[(6.2)]{GP11}).
\begin{equation}\label{err.et01}
(\E_t,\bphi_h)+\mu a(\E,\bphi_h)+\int_0^t \beta(t-s) a(\E(s),\bphi_h)~ds= \Lambda_{1,h}(\bphi_h)+(p,\nabla\cdot\bphi_h), ~~\forall\bphi_h\in \bJ_h,
\end{equation}
where
\begin{align}\label{lamb1}
\Lambda_{1,h}(\bphi_h) = b(\bu_h,\bu_h,\bphi_h)-b(\bu,\bu,\bphi_h)= -b(\bu_h,\E,\bphi_h)-b(\E,\bu,\bphi_h).
\end{align}
We differentiate (\ref{err.et01}) to find
\begin{align}\label{err.et02}
(\E_{tt},\bphi_h)+\mu a(\E_t,\bphi_h)+\gamma a(\E,\bphi_h)+\int_0^t \beta_t(t-s) a(\E(s),\bphi_h)~ds= \Lambda_{1,h,t}(\bphi_h)+(p_t,\nabla\cdot\bphi_h),
\end{align}
where
\begin{align}\label{lamb1t} 
\Lambda_{1,h,t}(\bphi_h)= -b(\bu_{h,t},\E,\bphi_h)-b(\bu_h,\E_t,\bphi_h) -b(\E_t,\bu,\bphi_h)-b(\E,\bu_t,\bphi_h).
\end{align}
With $\bv_h$ as the solution of a linearized Galerkin approximation $(5.2)$ of \cite{GP11}, we break the error $\E_t$ in linear and non-linear parts, as
$$ \E_t=(\bu-\bv_h)_t+(\bv_h-\bu_h)_t := \bxi_t+\bta_t. $$
The equation in $\bxi_t$ is given by
\begin{align}\label{bxit}
(\bxi_{tt},\bphi_h)+\mu a(\bxi_t,\bphi_h)+\gamma a(\bxi,\bphi_h)+\int_0^t \beta_t(t-s)
a(\bxi(s),\bphi_h)~ds= (p_t,\nabla\cdot\bphi_h).
\end{align}
With $P_h:\bL^2(\Omega)\to \bJ_h$ as $L^2$-projection, we choose $\bphi_h= \sigma_1(t) P_h\bxi_t=\sigma_1(t)(\bxi_t-(\bu-P_h\bu)_t)$ in (\ref{bxit}) to find (here $\sigma_1(t)= (\tau^*(t))^2 e^{2\alpha t},~\tau^*(t)=\min \{1,t\}$)
\begin{align}\label{err.et04a}
\frac{1}{2}\frac{d}{dt}\big\{\sigma_1(t)\|\bxi_t\|^2\}+\mu\sigma_1(t)\|\bxi_t\|_1^2=
\frac{1}{2}\sigma_{1,t}(t)\|\bxi_t\|^2-\gamma\sigma_t(t) a(\bxi,\bxi_t-(\bu-P_h\bu)_t) \nonumber \\
+\sigma_1(t)(\bxi_{tt},(\bu-P_h\bu)_t)+\mu\sigma_1(t) a(\bxi_t,(\bu-P_h\bu)_t) \\
-\sigma_1(t)\int_0^t \beta_t(t-s) a(\bxi(s),\bxi_t-(\bu-P_h\bu)_t)~ds
+\sigma_1(t)(p_t,\nabla\cdot P_h\bxi_t). \nonumber
\end{align}
 We note that
\begin{align*}
\sigma_1(t)(\bxi_{tt},(\bu-P_h\bu)_t) &=\sigma_1(t)((\bu-P_h\bu)_{tt},(\bu-P_h\bu)_t)
=\frac{\sigma_1(t)}{2}\frac{d}{dt}\|(\bu-P_h\bu)_t\|^2 \\
&=\frac{d}{dt}\big\{\frac{\sigma_1(t)}{2}\|(\bu-P_h\bu)_t\|^2\big\}-\frac{\sigma_{1,t}(t)}{2}
\|(\bu-P_h\bu)_t\|^2, \\
\sigma_1(t)(p_t,\nabla\cdot P_h\bxi_t) &=\sigma_1(t)(p_t-j_hp_t,\nabla\cdot P_h\bxi_t).
\end{align*}
Incorporate these in (\ref{err.et04a}). Then, use projection properties $({\bf B1})$ and Cauchy-Schwarz inequality to have
\begin{align}\label{bxit01}
\frac{1}{2}\frac{d}{dt}\big\{\sigma_1(t)\|\bxi_t\|^2\}+\mu\sigma_1(t)\|\bxi_t\|_1^2 \le
C\sigma(t)\|\bxi_t\|^2+C\sigma_1(t)\|\bxi\|_1\{\|\bxi_t\|_1+ch\|\bu_t\|_2\} \nonumber \\
\frac{d}{dt}\big\{\frac{\sigma_1(t)}{2}\|(\bu-P_h\bu)_t\|^2\big\}-\frac{\sigma_{1,t}(t)}{2}
\|(\bu-P_h\bu)_t\|^2+Ch~\sigma_1(t)\|\bxi_t\|_1\big\{\|\bu_t\|_2+\|p_t\|_1\big\} \nonumber \\
+C\sigma_1(t)\{\|\bxi_t\|_1+ch\|\bu_t\|_2\}\int_0^t \beta_t(t-s) \|\bxi(s)\|_1~ds.
\end{align}
($\sigma(t)=\tau^*(t) e^{2\alpha t},~\tau^*(t)=\min \{1,t\}$) \\
Use Young's inequality and kickback argument. Integrate the resulting inequality. We handle the double integral term as ($4.6$) of \cite{PY05} to find
\begin{align*}
\sigma_1(t)\|\bxi_t\|^2+\mu\int_0^t \sigma_1(s)\|\bxi_s(s)\|_1^2 ds \le C\int_0^t \sigma(s)\|\bxi_s(s)\|^2 ds+Ch^2\int_0^t \sigma_1(s) \big\{\|\bu_s(s)\|_2^2
+\|p_s(s)\|_1^2\big\}~ds \\
C\int_0^t e^{2\alpha s}\|\bxi(s)\|_1^2 ds+Ch~\sigma_1(t)\|\bu_t\|_1^2+Ch\int_0^t \sigma(s)
\|\bu_s(s)\|_1^2 ds.
\end{align*}
Use Theorem $3.2$ from \cite{GP11} to bound the terms involving $\bu$ and $p$. We borrow the estimate ($5.7$) from \cite{PY05} (although the estimate is for $\f\equiv 0$, can be achieved similarly in our case), that is,
\begin{equation}\label{err.et04}
e^{-2\alpha t}\int_0^t e^{2\alpha s}\|\bxi(s)\|_1^2 ds \le Kh^2
\end{equation}
to observe that
\begin{equation}\label{err.et05}
(\tau^*(t))^2\|\bxi_t\|^2+\mu e^{-2\alpha t}\int_0^t \sigma_1(s)\|\bxi_s(s)\|_1^2 ds \le
Ce^{-2\alpha t}\int_0^t \sigma(s)\|\bxi_s(s)\|^2 ds+Kh^2.
\end{equation}
To estimate the first term on the right hand-side of (\ref{err.et05}), we first recall the equation in $\bxi$ (see \cite[(5.3)]{GP11}):
\begin{align*}
(\bxi_t,\bphi_h)+\mu a(\bxi,\bphi_h)+\int_0^t \beta(t-s) a(\bxi(s),\bphi_h)~ds
=(p,\nabla\cdot\bphi_h),~~\forall\bphi_h\in \bJ_h.
\end{align*}
Choose $\bphi_h=\sigma(t)P_h\bxi_t$ above.
\begin{align*}
\sigma(t)\|\bxi_t\|^2+\frac{\mu}{2}\frac{d}{dt}\big\{\sigma(t)\|\bxi\|_1^2\} \le
Ce^{2\alpha t}\|\bxi\|_1^2+Ch.\sigma(t)\big\{\|\bxi_t\|\|\bu_t\|_1
+\|\bxi\|_1\|\bu_t\|_2+\|p\|_1\|\bxi_t\|_1\big\} \\
+C\sigma(t)\{\|\bxi_t\|_1+Ch\|\bu_t\|_2\}\int_0^t \beta_t(t-s) \|\bxi(s)\|_1~ds.
\end{align*}
Integrate with respect to time, use kickback argument, Lemma \ref{est1.uh} and (\ref{err.et04}) to conclude
\begin{equation}\label{err.et06}
e^{-2\alpha t}\int_0^t \sigma(s)\|\bxi_s(s)\|^2 ds+\tau^*(t)\|\bxi\|_1^2 \le Kh^2.
\end{equation}
We have again used Theorem $3.2$ from \cite{GP11} to bound $\bu_t$ and $p$.
Using (\ref{err.et06}) in (\ref{err.et05}) leads us to
\begin{equation}\label{err.et07}
(\tau^*(t))^2\|\bxi_t\|^2+ e^{-2\alpha t}\int_0^t \sigma_1(s)\|\bxi_s(s)\|_1^2 ds \le Kh^2.
\end{equation}
\noindent Next, we use parabolic duality argument (similar to the arguments ($5.8$)-($5.15$) and the rest, in \cite{PY05}) to establish 
\begin{equation}\label{err.et08}
e^{-2\alpha t}\int_0^t \sigma_1(s)\|\bxi_s(s)\|^2 ds \le Kh^4.
\end{equation}
An estimate of $\bta_t$ would now complete the proof. By definition, we can easily deduce the equation satisfied by $\bta_t$.
\begin{align}\label{btat}
(\bta_{tt},\bphi_h)+\mu a(\bta_t,\bphi_h)+\gamma a(\bta,\bphi_h)+\int_0^t \beta_t(t-s)
a(\bta(s),\bphi_h)~ds= \Lambda_{1,h,t}(\bphi_h),
\end{align}
where $\bphi_{1,h,t}$ is given by (\ref{lamb1t}). Choose $\bphi_h=\sigma_1(t)(-\td_h)^{-1}\bta_t$ and similar to (\ref{bxit01}), we find
\begin{align}\label{btat01}
\frac{1}{2}\frac{d}{dt}\big\{\sigma_1(t)\|\bta_t\|_{-1}^2\}+\mu\sigma_1(t)\|\bta_t\|^2 \le
C\sigma(t)\|\bta_t\|_{-1}^2+C\sigma_1(t)\|\bta\|\|\bta_t\| \nonumber \\
+\sigma_1(t)\|\bta_t\|\int_0^t \beta_t(t-s) \|\bta(s)\|~ds+\Lambda_{1,h,t}(\sigma_1(t)
(-\td_h)^{-1}\bta_t).
\end{align}
From (\ref{lamb1t}) and from the definition of the non-linear operator, we have
\begin{align}\label{err.et09}
\Lambda_{1,h,t}((-\td_h)^{-1}\bta_t)=& -b(\bu_{h,t},\E,(-\td_h)^{-1}\bta_t)-b(\bu_h,\E_t,
(-\td_h)^{-1}\bta_t) -b(\E_t,\bu,(-\td_h)^{-1}\bta_t) \nonumber \\
&-b(\E,\bu_t,(-\td_h)^{-1}\bta_t) \nonumber \\
= -\frac{1}{2}\big\{((\bu_{h,t}\cdot\nabla)&\E,(-\td_h)^{-1}\bta_t)-((\bu_{h,t}\cdot\nabla) (-\td_h)^{-1}\bta_t,\E)+((\bu_h\cdot\nabla)\E_t,(-\td_h)^{-1}\bta_t) \nonumber \\
-((\bu_h\cdot\nabla)&(-\td_h)^{-1}\bta_t, \E_t)+((\E_t\cdot\nabla)\bu,(-\td_h)^{-1}\bta_t) -((\E_t\cdot\nabla) (-\td_h)^{-1}\bta_t,\bu) \nonumber \\
+((\E\cdot\nabla)&\bu_t,(-\td_h)^{-1}\bta_t)-((\E\cdot\nabla)(-\td_h)^{-1}\bta_t,\bu_t)\big\}.
\end{align}
For the $2nd, ~7th$ and $8th$ terms on the right-hand side of (\ref{err.et09}), we use Lemmas \ref{nonlin} and \ref{est1.uh}, and Theorem $3.2$ from \cite{GP11}. For some $\varepsilon>0$, we now find
\begin{align}\label{err.et10}
&~\frac{1}{2}\sigma_1(t)\big\{((\bu_{h,t}\cdot\nabla)(-\td_h)^{-1}\bta_t,\E)-((\E\cdot\nabla)
\bu_t,(-\td_h)^{-1}\bta_t)+((\E\cdot\nabla)(-\td_h)^{-1}\bta_t,\bu_t)\big\} \nonumber \\
\le & ~C\sigma_1(t)\|\E\|\|\bta_t\|_{-1}^{1/2}\|\bta_t\|^{1/2}(\|\bu_{h,t}\|_1
+\|\bu_t\|_1)
\le \varepsilon \sigma_1(t)\|\bta_t\|^2+C\sigma(t)\|\bta_t\|_{-1}^2+Ke^{2\alpha t}\|\E\|^2.
\end{align}
Similarly, writing $\E=\bxi+\bta$ and using Lemmas \ref{nonlin} and \ref{est.uh}, we estimate the $4th,~5th$ and $6th$ terms on the right-hand side of (\ref{err.et09}).
\begin{align}\label{err.et11}
\frac{1}{2}\sigma_1(t)&~\big\{((\bu_h\cdot\nabla)(-\td_h)^{-1}\bta_t,\E_t)-((\E_t\cdot\nabla)
\bu,(-\td_h)^{-1}\bta_t)+((\E_t\cdot\nabla)(-\td_h)^{-1}\bta_t,\bu)\big\} \nonumber \\
\le &~C\sigma_1(t)(\|\bxi_t\|+\|\bta_t\|)\|\bta_t\|_{-1}^{1/2}\|\bta_t\|^{1/2}(\|\bu_h\|_1
+\|\bu\|_1)  \nonumber \\
\le &~K\sigma_1(t)\|\bxi_t\|\|\bta_t\|_{-1}^{1/2}\|\bta_t\|^{1/2}+K\sigma_1(t)
\|\bta_t\|_{-1}^{1/2}\|\bta_t\|^{3/2}  \nonumber \\
\le &~\varepsilon \sigma_1(t)\|\bta_t\|^2+K\sigma(t)\|\bta_t\|_{-1}^2+K\sigma_1(t)
\|\bxi_t\|^2.
\end{align}
To bound the $1st$ term on the right-hand side of (\ref{err.et09}), we first rewrite it as follows: \\
({\it with the notations $D_i=\frac{\partial}{\partial x_i}$ and $\bv=(v_1,v_2)$})
\begin{align}\label{nonlin02}
((\bu_{h,t}\cdot\nabla)\E, &(-\td_h)^{-1}\bta_t) =\sum_{i,j=1}^2 \int_{\Omega} u_{h,t,i}
D_i E_j((-\td_h)^{-1}\bta_t)_j~d{\bf x} \nonumber \\
&= -\sum_{i,j=1}^2 \int_{\Omega} D_i u_{h,t,i}E_j((-\td_h)^{-1}\bta_t)_j~d{\bf x} -\sum_{i,j=1}^2 \int_{\Omega} u_{h,t,i}D_i \{((-\td_h)^{-1}\bta_t)_j\}E_j~d{\bf x}
\nonumber \\
&= -((\nabla\cdot\bu_{h,t})\E,(-\td_h)^{-1}\bta_t)-((\bu_{h,t}\cdot\nabla)(-\td_h)^{-1} \bta_t,\E).
\end{align}
As in (\ref{err.et10}), we now have the following bound.
\begin{align}\label{err.et11a}
-\frac{1}{2}\sigma_1(t)((\bu_{h,t}\cdot\nabla)\E,(-\td_h)^{-1}\bta_t) \le \varepsilon \sigma_1(t)\|\bta_t\|^2+C\sigma(t)\|\bta_t\|_{-1}^2+Ke^{2\alpha t}\|\E\|^2.
\end{align}
Using similar technique and as in (\ref{err.et11}), we can bound the $3rd$ term on the right-hand side of (\ref{err.et09}).
\begin{align}\label{err.et11b}
-\frac{1}{2}\sigma_1(t)((\bu_h\cdot\nabla)\E_t,(-\td_h)^{-1}\bta_t) \le \varepsilon \sigma_1(t)\|\bta_t\|^2+K\sigma(t)\|\bta_t\|_{-1}^2+K\sigma_1(t)\|\bxi_t\|^2.
\end{align}
Incorporate (\ref{err.et10})-(\ref{err.et11b}) in (\ref{err.et09}) and then in (\ref{btat01}). Use Young's inequality and appropriate $\varepsilon>0$ and finally integrate with respect time to find
\begin{align*}
\sigma_1(t)\|\bta_t\|_{-1}^2+\mu\int_0^t \sigma_1(s)\|\bta_s(s)\|^2 ds \le
K\int_0^t \sigma(s)\|\bta_s(s)\|_{-1}^2 ds \nonumber \\
+K\int_0^t e^{2\alpha s}\big\{\|\bta(s)\|^2+\|\E(s)\|^2\big\}~ds+K\int_0^t \sigma_1(s)
\|\bxi_s(s)\|^2 ds.
\end{align*}
Use Lemma $5.4$ from \cite{GP11} and (\ref{err.et08}) to obtain
\begin{align}\label{err.et12}
\sigma_1(t)\|\bta_t\|_{-1}^2+\mu\int_0^t \sigma_1(s)\|\bta_s(s)\|^2 ds \le K(t)e^{2\alpha t}h^4+K\int_0^t \sigma(s)\|\bta_s(s)\|_{-1}^2 ds.
\end{align}
To estimate the last term of (\ref{err.et12}), we recall the equation in $\bta$ (see (5.12) of \cite{GP11}):
\begin{align}\label{bta0}
(\bta_t,\bphi_h)+\mu a(\bta,\bphi_h)+\int_0^t \beta(t-s) a(\bta(s),\bphi_h)~ds= \Lambda_{1,h}(\bphi_h),
\end{align}
where $\Lambda_{1,h}$ is given by (\ref{lamb1}).
We rewrite it using integration by parts.
\begin{align}\label{bta}
(\bta_t,\bphi_h)+\mu a(\bta,\bphi_h)+\gamma a(\tbta,\bphi_h)+\int_0^t \beta_t(t-s) a(\tbta(s),
\bphi_h)~ds=\Lambda_{1,h}(\bphi_h),
\end{align}
where
$$ \tbta(t)= \int_0^t \bta(s)~ds. $$
Choose $\bphi_h=\sigma(t)(-\td_h)^{-1}\bta_t$ in (\ref{bta}).
\begin{align}\label{err.et13}
\sigma(t)\|\bta_t\|_{-1}^2+\frac{\mu}{2}\frac{d}{dt}\{\sigma(t)\|\bta\|^2\} \le 
\frac{\mu}{2}\sigma_t(t)\|\bta\|^2+C\sigma(t)\|\bta_t\|_{-1}\big\{\|\tbta\|_1+\int_0^t \beta(t-s) \|\tbta(s)\|_1~ds\big\} \nonumber \\
+\sigma(t)\Lambda_{1,h}((-\td_h)^{-1}\bta_t).
\end{align}
We estimate the non-linear term $\Lambda_{1,h}$, following the technique used to estimate (\ref{err.et09}).
\begin{align*}
\sigma(t)\Lambda_{1,h}((-\td_h)^{-1}\bta_t)= -\sigma(t) b(\E,\bu_h,(-\td_h)^{-1}\bta_t)
-\sigma(t) b(\bu, \E,(-\td_h)^{-1}\bta_t) \\
\le c\sigma(t)\|\E\|\|\bta_t\|_{-1}\{\|\bu_h\|_2+\|\bu\|_2\}
\end{align*}
Incorporate this in (\ref{err.et13}) and use kickback argument. We bound $\bu_h$ and $\bu$ using Lemma \ref{est.uh} and Theorem $3.1$ from \cite{GP11}, and then integrate with respect to time.
\begin{align}\label{err.et14}
\int_0^t \sigma(s)\|\bta_s(s)\|_{-1}^2 ds+\sigma(t)\|\bta\|^2 \le 
c\int_0^t e^{2\alpha s}\big\{\|\bta(s)\|^2+\|\tbta(s)\|_1^2+K\|\E(s)\|^2\big\}~ds.
\end{align}
Apply Lemmas $5.4$ and $5.5$ from \cite{GP11} to bound the left-hand side of (\ref{err.et14}) by $K(t)h^4$. Incorporate this in (\ref{err.et12}) to complete the rest of the proof.
\end{proof}

\begin{lemma}\label{err.so} 
Under the assumptions of Theorem \ref{err.et}, the following result holds:
\begin{equation}\label{err.so1}
e^{-2\alpha t}\int_0^t e^{2\alpha s}\|(\bu-\bu_h)(s)\|_1^2 ds+e^{-2\alpha t}\int_0^t \sigma_1(s)\|(\bu_s-\bu_{h,s})(s)\|_1^2 ds \le K(t)h^2.
\end{equation}
\end{lemma}

\begin{proof}
Note that $\bu-\bu_h=\E= \bxi-\bta$. And so, keeping in mind (\ref{err.et04}) and (\ref{err.et07}), we only need estimates of $\bta$. For that, we choose
$\bphi_h=e^{2\alpha t}\bta$ in (\ref{bta0}).
\begin{align}\label{err.so01}
\frac{1}{2}\frac{d}{dt}\{e^{2\alpha t}\|\bta\|^2\}+\mu e^{2\alpha t}\|\bta\|_1^2
+\int_0^t \beta(t-s) a(\bta(s), e^{2\alpha t}\bta)~ds =\alpha e^{2\alpha t}\|\bta\|^2
+e^{2\alpha t}\Lambda_{1,h}(\bta).
\end{align}
The non-linear term is estimated, similar to (\ref{err.et11}).
\begin{align}\label{err.so02}
\Lambda_{1,h}(\bta)= -b(\bxi+\bta,\bu_h,\bta)-b(\bu,\bxi,\bta) \le \varepsilon \|\bta\|_1^2 +K(\|\bxi\|_1^2+\|\bta\|^2),
\end{align}
for some $\varepsilon>0$. Incorporate (\ref{err.et02}) in (\ref{err.so01}) and integrate. Drop the double integral term, following Lemma \ref{pp}.
\begin{align*}
e^{2\alpha t}\|\bta\|^2+\int_0^t e^{2\alpha s}\|\bta(s)\|_1^2 ds \le K\int_0^t e^{2\alpha s}
\|\bta(s)\|^2 ds+K\int_0^t e^{2\alpha s} \|\bxi(s)\|_1^2 ds.
\end{align*}
Use (\ref{err.et04}) and then apply Gronwall's Lemma to find
\begin{align}\label{err.so02a}
\|\bta\|^2+e^{-2\alpha t}\int_0^t e^{2\alpha s}\|\bta(s)\|_1^2 ds \le K(t)h^2.
\end{align}
This along with (\ref{err.et04}) completes the first part of (\ref{err.so1}). \\
Next, we choose $\bphi_h= \sigma_1(t)\bta_t$ in (\ref{btat}) to obtain as in (\ref{btat01})
\begin{align}\label{err.so03}
\frac{1}{2}\frac{d}{dt}\big\{\sigma_1(t)\|\bta_t\|^2\}+\mu\sigma_1(t)\|\bta_t\|_1^2 \le
C\sigma(t)\|\bta_t\|^2+C\sigma_1(t)\|\bta\|_1\|\bta_t\|_1 \nonumber \\
+\sigma_1(t)\|\bta_t\|_1\int_0^t \beta_t(t-s) \|\bta(s)\|_1~ds+\Lambda_{1,h,t}(\sigma_1(t)
\bta_t).
\end{align}
For the non-linear term $\Lambda_{1,h,t}$, we recall (\ref{err.et09}), except now, $(-\td_h)^{-1}\bta_t$ is replaced by $\bta_t$. Terms not involving $E_t$ can be bounded as:
(using Theorem \ref{errest})
\begin{align*}
\le C\big(\|\bu_{h,t}\|_1+\|\bu_t\|_1\big)\|\E\|_1\|\bta_t\|_1 \le K(t)(\tau^*(t))^{-1/2}h
\big(\|\bu_{h,t}\|_1+\|\bu_t\|_1\big)\|\bta_t\|_1.
\end{align*}
And for terms involving $E_t$, we rewrite it as $\bxi_t+\bta_t$ and bound those terms as
\begin{align*}
\le C\|\bu_h\|_1(\|\bxi_t\|_1+\|\bta_t\|_1)\|\bta_t\|^{1/2}\|\bta_t\|_1^{1/2}
+C\|\bu_h\|_1\|\bta_t\|_1(\|\bxi_t\|_1+\|\bta_t\|^{1/2}\|\bta_t\|_1^{1/2}) \\
+C(\|\bxi_t\|_1+\|\bta_t\|^{1/2}\|\bta_t\|_1^{1/2})\|\bu\|_1\|\bta_t\|^{1/2}\|\bta_t\|_1^{1/2}
+C(\|\bxi_t\|_1+\|\bta_t\|^{1/2}\|\bta_t\|_1^{1/2})\|\bta_t\|_1\|\bu\|_1 \\
\le K\big(\|\bxi_t\|_1\|\bta_t\|_1+\|\bta_t\|\|\bta_t\|_1
+\|\bta_t\|^{1/2}\|\bta_t\|_1^{3/2}\big).
\end{align*}
Use Lemma \ref{est.uh} and Lemma $3.4$ from \cite{GP11} to bound $\|\bu_h\|_1$ and $\|\bu\|_1$. Now, incorporate the above two estimates in $\Lambda_{1,h,t}(\bta_t)$ and then
in (\ref{err.so03}). Use Young's inequality and kickback argument to find
\begin{align*}
\frac{d}{dt}\big\{\sigma_1(t)\|\bta_t\|^2\}+\mu\sigma_1(t)\|\bta_t\|_1^2 \le
C\sigma(t)\|\bta_t\|^2+C\sigma_1(t)\|\bta\|_1^2+K(t)h^2 \nonumber \\
+K\sigma_1(t)\big(\|\bxi_t\|_1^2+\|\bta_t\|^2\big)+K(t)h^2\sigma(t)\big(\|\bu_{h,t}\|_1^2 +\|\bu_t\|_1^2\big).
\end{align*}
Integrate with respect to time to get
\begin{align}\label{err.so04}
\sigma_1(t)\|\bta_t\|^2+\mu\int_0^t \sigma_1(s)\|\bta_s(s)\|_1^2 ds \le K\int_0^t \sigma(s)
\|\bta_s(s)\|^2~ds+C\int_0^t \sigma_1(s)(\|\bta(s)\|_1^2+\|\bxi_s(s)\|_1^2)~ ds \nonumber \\
+K(t)h^2+K(t)h^2\int_0^t \sigma(s)\big(\|\bu_{h,s}(s)\|_1^2++\|\bu_s(s)\|_1^2\big)~ds
\end{align}
Use Lemma \ref{est1.uh} and Theorem $3.2$ from \cite{GP11} to bound the last the term of (\ref{err.so04}). Next, we use (\ref{err.et07}) and (\ref{err.so02a}) to infer from (\ref{err.so04})
\begin{align}\label{err.so05}
\sigma_1(t)\|\bta_t\|^2+\mu\int_0^t \sigma_1(s)\|\bta_s(s)\|_1^2 ds \le K\int_0^t \sigma(s)
\|\bta_s(s)\|^2~ds+K(t)h^2.
\end{align}
To estimate the term on the right-hand side of (\ref{err.so05}), we put $\bphi_h=\sigma(t)
\bta_t$ in (\ref{bta0}).
\begin{align}\label{err.so06}
\sigma(t)\|\bta_t\|^2+\frac{\mu}{2}\frac{d}{dt}\big\{\sigma(t)\|\bta\|_1^2\big\}= 
\frac{1}{2}\sigma_t(t)\|\bta\|_1^2-\sigma(t)\int_0^t \beta(t-s) a(\bta(s),\bta_t)~ds +\sigma(t) \Lambda_{1,h}(\bta_t).
\end{align}
We first rewrite the integral term as:
\begin{align}\label{int}
\sigma(t)\int_0^t \beta(t-s) a(\bta(s),\bta_t)~ds= a(\bta_t, \sigma(t)\int_0^t \beta(t-s)
\bta(s)~ds) \nonumber \\
=\frac{d}{dt}\big[a(\bta, \sigma(t)\int_0^t \beta(t-s)\bta(s)~ds)\big]-a(\bta, \sigma_t(t)
\int_0^t \beta(t-s)\bta(s)~ds) \nonumber \\
-a(\bta, \sigma(t)\beta(0)\bta)-a(\bta, \sigma(t)\int_0^t \beta_t(t-s)\bta(s)~ds).
\end{align}
We estimate the non-linear term as (following the idea (\ref{nonlin02}) to avoid $\|\bta_t\|_1$)
$$ \Lambda_{1,h}(\bta_t) \le K\|\bta_t\|\|\E\|_1(\|\bu_h\|_2^{1/2}+\|\bu\|_2^{1/2}). $$
Incorporate this along with (\ref{int}) in (\ref{err.so06}). Use Young's inequality and kickback argument. Integrate with respect to time to find
\begin{align}\label{err.so07}
\sigma(t)\|\bta\|_1^2+\int_0^t \sigma(s)\|\bta_s(s)\|^2 ds \le C\int_0^t e^{2\alpha s}
\|\bta(s)\|_1^2 ds+K\int_0^t e^{2\alpha s} \|\E(s)\|_1^2 ds.
\end{align}
Use the first part of (\ref{err.so1}) and (\ref{err.so02a}) to bound the left-hand side
of (\ref{err.so07}) by $K(t)h^2$. Incorporate this in (\ref{err.so07}) and then in 
(\ref{err.so05}) to complete the rest of the proof.
\end{proof}

\section{Two-Level Finite Element Method}
\se

In this section, we work with another space discretizing parameter $H$, that corresponds to a coarse mesh. In other words, $0<h<H$ and both $h,H$ tend to $0$. We introduce associated conforming finite element spaces $(\bH_H,L_H)$ and $(\bH_h,L_h)$ such that $(\bH_H,L_H) \subset (\bH_h,L_h)$. And this two-level finite element is to find a pair $(\bu^h,p^h)$ as follows:

{\bf First Level}: We compute the mixed finite element approximation $(\bu_H,p_H)\in (\bH_H,L_H)$ of $(\bu,p)$ of (\ref{wfh}). In other words, we solve the nonlinear problem on a coarse mesh. Find $(\bu_H,p_H)\in (H_H,p_H)$ satisfying
\begin{align}\label{2lvlH1}
(\bu_{Ht}, \bphi_H) +\mu a (\bu_H,\bphi_H)+& b(\bu_H,\bu_H,\bphi_H)+ \int_0^t \beta(t-s) a(\bu_H(s),\bphi_H)~ds -(p_H, \nabla \cdot \bphi_H) =(\f, \bphi_H), \nonumber \\
&(\nabla \cdot \bu_H, \chi_H) =0,
\end{align}
for $(\bphi_H, \chi_H)\in (\bH_H, L_H)$.

{\bf Second Level}: We solve a linearized problem on a fine mesh. Find $(\bu^h,p^h) \in (\bH_h,L_h)$ satisfying
\begin{align}\label{2lvlH2}
(\bu^h_t, \bphi_h) +\mu a (\bu^h,\bphi_h)+& b(\bu_H,\bu_H,\bphi_h)+ \int_0^t \beta(t-s) a(\bu^h(s),\bphi_h)~ds -(p^h, \nabla \cdot \bphi_h) =(\f, \bphi_h), \nonumber \\
&(\nabla \cdot \bu^h, \chi_h) =0,
\end{align}
for $(\bphi_h, \chi_h)\in (\bH_h, L_h)$.

\noindent An equivalent way is look for solution in a weekly divergent free space.

{\bf First Level}: Find $\bu_H \in \bJ_H$ satisfying
\begin{align}\label{2lvlJ1}
(\bu_{Ht}, \bphi_H) +\mu a (\bu_H,\bphi_H)+ b(\bu_H,\bu_H,\bphi_H)+ \int_0^t \beta(t-s) a(\bu_H(s),\bphi_H)~ds =(\f, \bphi_H),
\end{align}
for $\bphi_H \in \bJ_H$.

{\bf Second Level}: 
With $\bu_H$ being the solution of (\ref{2lvlJ1}), find $\bu^h \in \bJ_h$ satisfying
\begin{align}\label{2lvlJ2}
(\bu^h_t, \bphi_h) +\mu a (\bu^h,\bphi_h)+ b(\bu_H,\bu_H,\bphi_h)+ \int_0^t \beta(t-s) a(\bu^h(s),\bphi_h)~ds =(\f, \bphi_h),
\end{align}
for $\bphi_h \in \bJ_h$.

\begin{remark}
The well-posedness of the above systems can be seen from the facts that (\ref{2lvlH1}) or (\ref{2lvlJ1}) is Galerkin finite element approximation and hence is well-posed as is stated in Section $3$. And (\ref{2lvlH2}) or (\ref{2lvlJ2}) represent linearized version. Given $\bu_H$ and with suitable $\bu^h(0)$, it is therefore an easy task to show the existence of an unique solution pair $(\bu^h,p^h)$ (or an unique solution $\bu^h$) for the linearized problem following the foot-steps of the non-linear problem.
\end{remark}
\noindent
Following Lemma \ref{est.uh}, we can easily obtain the {\it a priori} estimates of $\bu_H$.
\begin{lemma}\label{est.uH}
Under the assumptions of Lemma \ref{est.uh} and for $\bu_H(0)=P_H\bu_0$, the solution 
$\bu_H$ of (\ref{2lvlJ1}) satisfies, for $t>0,$
\begin{eqnarray}
\|\bu_H(t)\|+e^{-2\alpha t}\int_0^t e^{2\alpha s}\|\bu_H(t)\|_1^2~ds \le K, \label{uH01} \\
\|\bu_H(t)\|_1+e^{-2\alpha t}\int_0^t e^{2\alpha s}\|\bu_H(t)\|_2^2~ds \le K, \label{uH02} \\
(\tau^*(t))^{1/2}\|\bu_H(t)\|_2 \le K, \label{uH03}
\end{eqnarray}
where $\tau^*(t)= \min \{1,t\}$ and $K$ depends only on the given data. In particular, $K$ is independent of $H$ and $t$.
\end{lemma}
\noindent The following higher-order estimate of $\bu_H$ is required for error analysis. The proof of the same is similar to that of Lemma \ref{est1.uh}.
\begin{lemma}\label{est1.uH}
Under the assumptions of Lemma \ref{est.uH}, the solution $\bu_H$ of (\ref{2lvlJ1}) satisfies, for $t>0,$
\begin{equation}\label{1uH01}
(\tau^*(t))^2\|\bu_{H,t}(t)\|_1^2+(\tau^*(t))^3\|\bu_{H,t}(t)\|_2^2+e^{-2\alpha t}\int_0^t
e^{2\alpha s}(\tau^*(s))\|\bu_{H,s}(s)\|_1^2 ds \le K.
\end{equation}
\end{lemma}
\noindent 
Below, we deal with the {\it a priori} estimates of $\bu^h$.
\begin{lemma}\label{est.2uh}
Under the assumptions of Lemma \ref{est.uH}, the solution $\bu^h$ of (\ref{2lvlJ2}) satisfies, for $t>0,$
\begin{eqnarray}
\|\bu^h(t)\|+e^{-2\alpha t}\int_0^t e^{2\alpha s}\|\bu^h(t)\|_1^2~ds \le K, \label{2uh01} \\
\|\bu^h(t)\|_1+e^{-2\alpha t}\int_0^t e^{2\alpha s}\|\bu^h(t)\|_2^2~ds \le K, \label{2uh02}
\end{eqnarray}
\end{lemma}

\begin{proof}
Given $\bu_H$ with Lemma (\ref{est.uH}), we choose $\bphi_h= e^{2\alpha t}\bu^h(t)= e^{\alpha t} \hbu^h(t)$ in (\ref{2lvlJ2}) to obtain
\begin{align}\label{2uh001}
\frac{1}{2}\frac{d}{dt}\|\hbu^h\|^2-\alpha\|\hbu^h\|^2+\mu\|\hbu^h\|_1^2+\gamma\int_0^t e^{-(\delta-\alpha)(t-s)} a(\hbu^h(s),\hbu^h)~ds \nonumber \\
=(\hf,\hbu^h)-e^{2\alpha t} b(\bu_H,\bu_H,\bu^h).
\end{align}
Use Cauchy-Schwarz inequality, Poincar\'e inequality with first eigenvalue of Stokes operator as the constant and Young's inequality to have
\begin{align}\label{2uh002}
(\hf,\hbu^h) \le \|\hf\|\|\hbu^h\| \le \frac{1}{\lambda_1^{1/2}}\|\hf\|\|\hbu^h\|_1 
\le \frac{\mu}{4}\|\hbu^h\|_1^2+\frac{1}{\mu\lambda_1}\|\hf\|^2.
\end{align}
From Lemma \ref{nonlin} and Lemma \ref{est.uH}, we find that
\begin{align*}
b(\bu_H,\bu_H,\bu^h) & \le \|\bu_H\|^{1/2}\|\bu_H\|_1^{3/2}\|\bu_h\|^{1/2}\|\bu_h\|_1^{1/2}
+\|\bu_H\|\|\bu_H\|_1\|\bu_h\|_1 \nonumber \\
& \le K+\frac{\mu}{4}\|\bu^h\|_1^2
\end{align*}
Therefore,
\begin{equation}\label{2uh003}
e^{2\alpha t} b(\bu_H,\bu_H,\bu^h) \le Ke^{2\alpha t}+\frac{\mu}{4}\|\hbu^h\|_1^2.
\end{equation}
Putting the estimates (\ref{2uh002})-(\ref{2uh003}) in (\ref{2uh001}) gives us
\begin{align*}
\frac{d}{dt}\|\hbu^h\|^2-2\alpha\|\hbu^h\|^2+\mu\|\hbu^h\|_1^2+2\gamma\int_0^t e^{-(\delta-\alpha)(t-s)} a(\hbu^h(s),\hbu^h)~ds \le Ke^{2\alpha t}+\frac{1}{\mu\lambda_1} \|\hf\|^2.
\end{align*}
Use Poincar\'e inequality to get
\begin{align}\label{2uh004}
\frac{d}{dt}\|\hbu^h\|^2+\big(\mu-\frac{2\alpha}{\lambda_1}\big)\|\hbu^h\|_1^2+2\gamma
\int_0^t e^{-(\delta-\alpha)(t-s)} a(\hbu^h(s),\hbu^h)~ds \le Ke^{2\alpha t} +\frac{1}{\mu\lambda_1} \|\hf\|^2.
\end{align}
Since $0<\alpha< \min\{\mu\lambda_1/2,\delta\}$, we have that $\mu-2\alpha/\lambda_1>0$. \\
Integrate (\ref{2uh004}) with respect to time and use the positivity of the kernel $\beta$, that is, Lemma \ref{pp}, to drop the resulting double integral. This results in
\begin{align*}
\|\hbu^h(t)\|^2+\big(\mu-\frac{2\alpha}{\lambda_1}\big)\int_0^t \|\hbu^h(s)\|_1^2 ds
\le Ke^{2\alpha t}.
\end{align*}
Multiply by $e^{-2\alpha t}$ to conclude (\ref{2uh01}). \\
For the next estimate, we choose $\bphi_h= e^{2\alpha t}\td_h\bu^h(t)$ in (\ref{2lvlJ2}) and proceed as above. For the non-linear term, by the definition of the operator b, we have
$$ b(\bu_H,\bu_H,\td_h\bu^h)=\frac{1}{2}((\bu_H\cdot\nabla)\bu_H,\td_h\bu^h)-\frac{1}{2}
((\bu_H\cdot\nabla)\td_h\bu^h,\bu_H). $$
Use Lemma \ref{nonlin} to obtain
\begin{align}\label{2uh005}
\frac{1}{2}((\bu_H\cdot\nabla)\bu_H,\td_h\bu^h) \le C\|\bu_H\|^{1/2}\|\bu_H\|_2^{1/2}
\|\bu_H\|_1\|\bu^h\|_2.
\end{align}
Following the idea (\ref{nonlin02}), we have
\begin{align}\label{2uh006}
((\bu_H\cdot\nabla)\td_h\bu^h,\bu_H)
&= -((\nabla\cdot\bu_H)\bu_H,\td_h\bu^h)-((\bu_H\cdot\nabla)\bu_H,\td_h\bu^h) \nonumber \\
&\le C\|\bu_H\|^{1/2}\|\bu_H\|_2^{1/2}\|\bu_H\|_1\|\bu^h\|_2.
\end{align}
These estimates let us have, with $\bphi_h= e^{2\alpha t}\td_h\bu^h(t)$ in (\ref{2lvlJ2}):
\begin{align}\label{2uh007}
\frac{1}{2}\frac{d}{dt}\|\hbu^h\|_1^2-\alpha\|\hbu^h\|_1^2+\mu\|\hbu^h\|_2^2+\gamma\int_0^t e^{-(\delta-\alpha)(t-s)} (\td_h\hbu^h(s),\td_h\hbu^h)~ds \nonumber \\
\le \|\hf\|\|\hbu^h\|_2+Ce^{\alpha t} \|\bu_H\|^{1/2}\|\bu_H\|_2^{1/2}\|\bu_H\|_1\|\hbu^h\|_2.
\end{align}
Using Young's inequality and Poincar\'e inequality as earlier, we find that
\begin{align}\label{2uh008}
\frac{d}{dt}\|\hbu^h\|_1^2+\big(\mu-\frac{\alpha}{2\lambda_1}\big)\|\hbu^h\|_2^2+2\gamma
\int_0^t e^{-(\delta-\alpha)(t-s)} (\td_h\hbu^h(s),\td_h\hbu^h)~ds \nonumber \\
\le C\|\hf\|^2+ce^{2\alpha t}\big\{\|\bu_H\|^2\|\bu_H\|_1^4+\|\bu_H\|_2^2\big\}.
\end{align}
Integrate (\ref{2uh008}) with respect to time, use Lemma \ref{est.uH} and finally multiply by $e^{-2\alpha t}$ to obtain (\ref{2uh02}).
\end{proof}

\section{Error Estimate}
\se

In this section, we present the error estimate for the spatial approximation, that is, two-level finite element Galerkin approximation. We achieve our desired results through a series of Lemmas. We denote the error, due to two-level method, as $\e=\bu_h-\bu^h$.

\noindent
From the equations (\ref{dwfj}) and (\ref{2lvlJ2}), we have the following error equation:
\begin{align}\label{err}
(\e_t,\bphi_h) +\mu a (\e,\bphi_h)+ \int_0^t \beta(t-s) a(\e(s), \bphi_h)~ds
= \Lambda_h(\bphi_h)~~\forall \bphi_h \in \bJ_h,
\end{align}
where
\begin{align}\label{lamb}
\Lambda_h(\bphi_h) & = b(\bu_H,\bu_H,\bphi_h)-b( \bu_h, \bu_h, \bphi_h) \nonumber \\
&= b(\bu_H,\bu_H-\bu_h,\bphi_h)+b(\bu_H-\bu_h,\bu_h,\bphi_h).
\end{align}

\begin{lemma}\label{neg} 
Under the assumptions of Theorem \ref{errest} and with the additional assumption that $\bu^h(0)=\bu_h(0)$ and for $0<\alpha< \min \{\mu\lambda_1/2,\delta\}$, the following estimate
\begin{equation}\label{neg1}
\|\e(t)\|_{-1}^2+e^{-2\alpha t}\int_0^t e^{2\alpha s}\|\e(s)\|^2 ds \le K(t)H^4
\end{equation}
holds, for $t>0$.
\end{lemma}

\begin{proof}
Choose $\bphi_h=e^{2\alpha t}(-\td_h)^{-1}\e(t)= e^{\alpha t}(-\td_h)^{-1}\he(t)$ in (\ref{err}) to obtain
\begin{equation}\label{neg01}
\frac{1}{2}\frac{d}{dt}\|\he\|_{-1}^2-\alpha\|\he\|_{-1}^2+\mu\|\he\|^2+\gamma \int_{t_0}^t
e^{-(\delta-\alpha)(t-s)} (\he(s),\he)~ds= e^{2\alpha t}\Lambda_h((-\td_h)^{-1}\e).
\end{equation}
From (\ref{lamb}) and Lemmas \ref{nonlin}, \ref{est.uh} and \ref{est.uH}, we have
\begin{align*}
\Lambda((-\td_h)^{-1}\e) &= b(\bu_H,\bu_H-\bu_h,(-\td_h)^{-1}\e)+b(\bu_H-\bu_h,\bu_h,(-\td_h)^{-1}\e) \\
& \le C\|\bu_H-\bu_h\|\|\e\|\big\{\|\bu_H\|_1+\|\bu_h\|_1\big\}
\le K\|\bu_H-\bu_h\|\|\e\|.
\end{align*}
We have again used the idea (\ref{nonlin02}).
Incorporate this in (\ref{neg01}). Use Young's inequality and kickback argument.
\begin{equation}\label{neg02}
\frac{d}{dt}\|\he\|_{-1}^2+\mu\|\he\|^2+2\gamma \int_{t_0}^t e^{-(\delta-\alpha)(t-s)} (\he(s),\he)~ds \le 2\alpha\|\he\|_{-1}^2+K\|\hbu_H-\hbu_h\|^2.
\end{equation}
Integrate (\ref{neg02}) with respect to time. The resulting double integral is dropped following Lemma \ref{pp}. Observe that
\begin{align}\label{est.hH}
\int_0^t \|(\hbu_H-\hbu_h)(s)\|^2 ds \le \int_0^t \|(\hbu-\hbu_H)(s)\|^2 ds
+\int_0^t \|(\hbu-\hbu_h)(s)\|^2 ds \le Ke^{2\alpha t}H^4.
\end{align}
We have used here the Lemma 5.4 of \cite{GP11}. And hence
\begin{equation}\label{neg03}
\|\he(t)\|_{-1}^2+\mu\int_0^t \|\he(s)\|^2 ds \le 2\alpha\int_0^t \|\he(s)\|_{-1}^2 ds +Ke^{2\alpha t}H^4.
\end{equation}
Multiply by $e^{-2\alpha t}$ and use Gronwall's Lemma to complete the rest of the proof.
\end{proof}

\begin{remark}
Due to non-smooth initial data, in the estimate of $\|\e(t)\|$, it is essential to introduce $\sigma(t)=\tau^*(t)e^{2\alpha t},~~\tau^*(t)=\min\{1,t\}$, in order to avoid nonlocal compatibility conditions (see \cite{GP11} and references cited there). However, the resulting integral term  (with $\bphi_h=\sigma(t)\e(t)$) is now no longer non-negative after integration. A direct estimate of this term is not possible. To handle this new difficulty, we present the following notation
$$ \te(t)=\int_0^t \e(s)~ds $$
and obtain an estimate for
$$ \int_0^t e^{2\alpha t} \|\te(s)\|_1^2 ds. $$
And with this intermediate estimate, we manage to estimate the integral term with no loss of rate of convergence.
\end{remark}

\noindent First, we integrate the error equation (\ref{err}) in time and observe that $\te_t(t)=\e(t)$, since $\e(0)=0.$ Also from (5.18) of \cite{GP11}, we find that the resulting double integral can be written in single integral form.
\begin{align}\label{erri} 
(\te_t,\bphi_h) +\mu a (\te,\bphi_h)+ \int_0^t \beta(t-s) a(\te(s), \bphi_h)~ds
= \int_0^t \Lambda_h(\bphi_h)~ds ~~\forall \bphi_h \in \bJ_h,
\end{align}

\begin{lemma}\label{inter} 
Under the assumptions of Lemma \ref{neg}, the following estimate
\begin{equation}\label{inter1}
\|\te(t)\|^2+e^{-2\alpha t}\int_0^t e^{2\alpha s} \|\te(s)\|_1^2 ds
\le Ke^{2\alpha t}t^{1/2}H^4
\end{equation}
holds, for $t>0$.
\end{lemma}

\begin{proof}
Put $\bphi_h=e^{2\alpha t}\te(t)$ in (\ref{erri}) to find
\begin{align}\label{inter01}
\frac{1}{2}\frac{d}{dt}\big\{e^{2\alpha t}\|\te\|^2\big\}-\alpha e^{2\alpha t}\|\te\|^2
+\mu e^{2\alpha t}\|\te\|_1^2+\gamma \int_0^t e^{-(\delta-\alpha)(t-s)} a(e^{\alpha s} \te(s),e^{\alpha t}\te)~ds \nonumber \\
= e^{2\alpha t} \int_0^t \Lambda_h(\te)~ds.
\end{align}
Now from (\ref{lamb}) and Lemma \ref{nonlin}, we have (following (\ref{nonlin02}))
\begin{align}\label{inter02}
\Lambda_h(\te) &= b(\bu_H,\bu_H-\bu_h,\te)+b(\bu_H-\bu_h,\bu_h,\te) \nonumber \\
& \le C\|\bu_H-\bu_h\|\|\te\|_1\big\{\|\bu_H\|_1^{1/2}\|\bu_H\|_2^{1/2}+\|\bu_h\|_1^{1/2}
\|\bu_h\|_2^{1/2}\big\}.
\end{align}
Incorporate (\ref{inter02}) to obtain from (\ref{inter01})
\begin{align*}
\frac{d}{dt}\big\{e^{2\alpha t}\|\te\|^2\big\}-2\alpha e^{2\alpha t}\|\te\|^2
+2\mu e^{2\alpha t}\|\te\|_1^2+2\gamma \int_{t_0}^t e^{-(\delta-\alpha)(t-s)} a(e^{\alpha s} \te(s),e^{\alpha t}\te)~ds \\
\le ce^{2\alpha t}\|\te\|_1 \int_0^t \|(\bu_H-\bu_h)(s)\|\big\{\|\bu_H(s)\|_1^{1/2} \|\bu_H(s)\|_2^{1/2} +\|\bu_h(s)\|_1^{1/2}\|\bu_h(s)\|_2^{1/2}\big\}~ds.
\end{align*}
Use Young's inequality, kickback argument and then Poincar\'e inequality. Again apply Young's inequality and then Lemmas \ref{est.uH} and \ref{est.uh} to find
\begin{align}\label{inter03}
& \frac{d}{dt}\big\{e^{2\alpha t}\|\te\|^2\big\}+\big(\mu-\frac{2\alpha}{\lambda_1}\big) e^{2\alpha t}\|\te\|_1^2+2\gamma \int_{t_0}^t e^{-(\delta-\alpha)(t-s)} a(e^{\alpha s} \te(s),e^{\alpha t}\te)~ds \nonumber \\
\le & Ce^{2\alpha t}\Big[\int_0^t \|(\bu_H-\bu_h)(s)\|\big\{\|\bu_H(s)\|_1^{1/2} \|\bu_H(s)\|_2^{1/2} +\|\bu_h(s)\|_1^{1/2}\|\bu_h(s)\|_2^{1/2}\big\}~ds \Big]^2 \nonumber \\
\le & Ke^{2\alpha t}\int_0^t \|(\bu_H-\bu_h)(s)\|^2 ds \int_0^t \|\big\{\|\bu_H(s)\|_2 +\|\bu_h(s)\|_2\big\}~ds \nonumber \\
\le & Ke^{2\alpha t}\int_0^t \|(\bu_H-\bu_h)(s)\|^2 ds \int_0^t s^{-1/2}ds
\le Ke^{2\alpha t}t^{1/2}\int_0^t \|(\bu_H-\bu_h)(s)\|^2 ds.
\end{align}
We use the fact that $1< e^{\alpha t}$ since $\alpha$ and $t$ are positive. Use (\ref{est.hH}) to get
\begin{align*}
\int_0^t \|(\bu_H-\bu_h)(s)\|^2 ds \le \int_0^t \|(\hbu_H-\hbu_h)(s)\|^2 ds
\le Ke^{2\alpha t}H^4.
\end{align*}
With this, we now integrate (\ref{inter03}) in time and drop the double integral, since it is non-negative, see Lemma \ref{pp}.
\begin{align}\label{inter04}
& e^{2\alpha t}\|\te(t)\|^2+\big(\mu-\frac{2\alpha}{\lambda_1}\big)\int_0^t e^{2\alpha s} \|\te(s)\|_1^2 ds \le Ke^{4\alpha t}t^{1/2}H^4.
\end{align}
Multiply (\ref{inter04}) by $e^{-2\alpha t}$ to complete the rest of the proof.
\end{proof}

\begin{lemma}\label{pee} 
Under the assumptions of Lemma \ref{neg}, the following estimate
\begin{equation}\label{pee1}
\tau^*(t)\|\e(t)\|^2+e^{-2\alpha t}\int_0^t \sigma(s)\|\e(s)\|_1^2 ds \le K(t)H^4
\end{equation}
holds, for $t>0$, where $\sigma(t)=\tau^*(t)e^{2\alpha t},~\tau^*(t)=\min\{1,t\}$.
\end{lemma}

\begin{proof}
We first recall that $\beta(t)=\gamma e^{-\delta t}$ and now use integration by parts to find
\begin{equation}\label{pee01a}
\int_0^t \beta(t-s) \e(s)~ds= \gamma \te(t)-\delta\int_0^t \beta(t-s)\te(s)~ds.
\end{equation}
Keeping this in mind, we choose $\bphi_h=\sigma(t)\e(t)$ in (\ref{err}) to obtain
\begin{align}\label{pee01}
\frac{d}{dt}\big\{\sigma(t)\|\e\|^2\big\}+2\mu~\sigma(t)\|\e\|_1^2= \sigma_t(t)\|\e\|^2
+2\sigma(t)\Lambda_h(\e)-2\gamma\sigma(t) a(\te,\e) \nonumber \\
+2\delta\sigma(t) \int_0^t \beta(t-s) a(\te(s),\e)~ds.
\end{align}
We repeat the non-linear estimate (\ref{inter02}).
$$ \Lambda_h(\e) \le C\|\bu_H-\bu_h\|\|\e\|_1\big\{\|\bu_H\|_1^{1/2}\|\bu_H\|_2^{1/2} +\|\bu_h\|_1^{1/2}\|\bu_h\|_2^{1/2}\big\}. $$
And hence, we find from (\ref{pee01})
\begin{align*}
\frac{d}{dt}\big\{\sigma(t)\|\e\|^2\big\}+2\mu~\sigma(t)\|\e\|_1^2 \le (1+2\alpha)\|\he\|^2
+2\sigma(t)\|\e\|_1\Big\{\gamma\|\te\|_1+\delta\int_0^t \beta(t-s)\|\te\|_1 ds\Big\} \\
+C\sigma(t)\|\e\|_1\|\bu_H-\bu_h\|\big\{\|\bu_H\|_1^{1/2}\|\bu_H\|_2^{1/2} +\|\bu_h\|_1^{1/2}\|\bu_h\|_2^{1/2}\big\}.
\end{align*}
Use Young's inequality and then kickback argument to yield
\begin{align}\label{pee02}
\frac{d}{dt}\big\{\sigma(t)\|\e\|^2\big\}+\mu~\sigma(t)\|\e\|_1^2 \le (1+2\alpha)\|\he\|^2
+C\sigma(t)\Big\{\gamma\|\te\|_1+\delta\int_0^t \beta(t-s)\|\te\|_1 ds\Big\}^2 \nonumber \\
+C\sigma(t)\|\bu_H-\bu_h\|^2\big\{\|\bu_H\|_1\|\bu_H\|_2 +\|\bu_h\|_1\|\bu_h\|_2\big\}.
\end{align}
Use Lemmas \ref{est.uH} and \ref{est.uh} to estimate the last term of (\ref{pee02}) as follows:
$$ C\sigma(t)\|\bu_H-\bu_h\|^2\big\{\|\bu_H\|_1\|\bu_H\|_2 +\|\bu_h\|_1\|\bu_h\|_2\big\}
 \le K\|\hbu_H-\hbu_h\|^2. $$
We have used the fact that $\tau^*(t) \le 1$. Incorporate this in (\ref{pee02}) and integrate in time. Repeat the technique of (\ref{est.hH}).
\begin{align}\label{pee03}
\sigma(t)\|\e(t)\|^2+\mu\int_0^t \sigma(s)\|\e(s)\|_1^2 ds \le C\int_0^t \Big\{\|\he(s)\|^2
+e^{2\alpha s}\|\te(s)\|_1^2\Big\}~ds+Ke^{2\alpha t}H^4.
\end{align}
Now use Lemmas \ref{neg} and \ref{inter} and then multiply the resulting inequality by $e^{-2\alpha t}$ to complete the rest of the proof.
\end{proof}

\begin{lemma}\label{inter10}
Under the assumptions of Lemma \ref{pee}, the following holds:
\begin{equation}\label{inter11}
(\tau^*(t))^{1/2}\|\te(t)\|_1 \le K(t)H^2.
\end{equation}
\end{lemma}

\begin{proof}
Put $\bphi_h=\sigma(t)\te_t(t)=\sigma(t)\e(t)$ in (\ref{erri}) to find
\begin{align}\label{inter101}
\sigma(t)\|\te_t\|^2+\frac{\mu}{2}\frac{d}{dt}\big\{\sigma(t)\|\te\|_1^2\big\}
-\frac{\mu}{2}\sigma_t(t)\|\te\|_1^2= -\sigma(t)\int_0^t \beta(t-s) a(\te(s),\e)~ds
+\sigma(t) \int_0^t \Lambda_h(\e)~ds.
\end{align}
As in (\ref{inter02}), we have, using Lemmas \ref{est.uH} and \ref{est.uh}
\begin{align}\label{inter102}
\sigma(t) \int_0^t \Lambda_h(\e)~ds \le K\sigma(t)\|\e\|_1\int_0^t \|\bu_H(s)-\bu_h(s)\| \big\{\|\bu_H(s)\|_2^{1/2}+\|\bu_h(s)\|_2^{1/2}\big\}~ds.
\end{align}
Put (\ref{inter102}) in (\ref{inter101}) and integrate with respect to time. Use Cauchy-Schwarz and Young's inequalities. For the resulting double integral term proceed as in (\ref{inter03}) and (\ref{inter04}) to observe that
\begin{align}\label{inter103}
\mu\sigma(t)\|\te\|_1^2+2\int_0^t \sigma(s)\|\te_s(s)\|^2ds \le c\int_0^t \e^{2\alpha s}
\|\te(s)\|_1^2ds+c\int_0^t \sigma(s)\|\e\|_1^2ds+Ke^{4\alpha t} t^{3/2}H^4.
\end{align}
Use Lemmas \ref{inter} and \ref{pee} to complete the rest of the proof.
\end{proof}
\begin{lemma}\label{eeu} 
Under the assumptions of Lemma \ref{neg}, the following estimate
\begin{equation}\label{eeu1}
(\tau^*(t))^2\|\e(t)\|_1^2+e^{-2\alpha t}\int_0^t \sigma_1(s)\|\e_s(s)\|^2 ds \le K(t)H^4
\end{equation}
holds, for $t>0$, where $\sigma_1(t)=(\tau^*(t))^2 e^{2\alpha t}$.
\end{lemma}

\begin{proof}
Choose $\bphi_h=\sigma_1(t)\e_t(t)$ in (\ref{err}) and incorporate (\ref{pee01a}) to obtain
\begin{align}\label{eeu01}
\sigma_1(t)\|\e_t\|^2+\frac{\mu}{2}\frac{d}{dt}\big\{\sigma_1(t)\|\e\|_1^2\big\}= \frac{\mu}{2}\sigma_{1,t}(t)\|\e\|_1^2+\sigma_1(t)\Lambda_h(\e_t)-\gamma\sigma_1(t)
 a(\te,\e_t) \nonumber \\
+\delta\sigma_1(t) \int_0^t \beta(t-s) a(\te(s),\e_t)~ds.
\end{align}
Multiply (\ref{eeu01}) by $2$ and integrate with respect to time. Use integration by parts to observe that
\begin{align}\label{eeu02}
& -2\gamma\int_0^t \sigma_1(s) a(\te(s),\e_s)~ds \nonumber \\
& = -2\gamma\sigma_1(t) a(\te,\e)+2\gamma\int_0^t \sigma_{1,s}(s) a(\te(s),\e)~ds
+2\gamma\int_0^t \sigma_1(s) \|\e(s)\|_1^2~ds \nonumber \\
& \le \frac{\mu}{4}\sigma_1(t)\|\e\|_1^2+c\sigma_1(t)\|\te\|_1^2
+C\int_0^t \sigma(s)\|\e(s)\|_1^2 ds+C\int_0^t \sigma(s)\|\te(s)\|_1^2 ds.
\end{align}
And
\begin{align}\label{eeu03}
& 2\delta\int_0^t \sigma_1(s) \int_0^s \beta(s-\tau) a(\te(\tau),\e_s)~d\tau~ds \nonumber \\
=& 2\delta\sigma_1(t)\int_0^t \beta(t-s) a(\te(s),e)~ds-2\delta\int_0^t \sigma_{1,s}(s) \int_0^s \beta(s-\tau) a(\te(\tau),\e(s)~d\tau~ds \nonumber \\
&+2\delta^2\int_0^t \sigma_1(s) \int_0^s \beta(s-\tau) a(\te(\tau),\e(s))~d\tau~ds
-2\delta\gamma\int_0^t \sigma_1(s) a(\te(s),\e(s))~ds \nonumber \\
\le & \frac{\mu}{4}\sigma_1(t)\|\e\|_1^2+C\int_0^t e^{2\alpha s}\|\te(s)\|_1^2~ds
+C\int_0^t \sigma(s)\|\e(s)\|_1^2 ds.
\end{align}
We have again used the fact that $\tau^*(t) \le 1$. And now, after integration of (\ref{eeu01}), we incorporate (\ref{eeu02})-(\ref{eeu03}) to find
\begin{align}\label{eeu04}
2\int_0^t \sigma_1(s)\|\e_s(s)\|^2 ds +\frac{2\mu}{3}\sigma_1(t)\|\e\|_1^2 & \le
C\int_0^t e^{2\alpha s}\|\te(s)\|_1^2 ds+C\int_0^t \sigma(s)\|\e(s)\|_1^2 ds \nonumber \\
& +C\sigma_1(t)\|\te\|_1^2+2\int_0^t \sigma_1(s)\Lambda_h(\e_s)~ds.
\end{align}
We rewrite the non-linear terms as follows:
\begin{align*}
\Lambda_h(\e_t) &=b(\bu_H,\bu_H-\bu_h,\e_t)+b(\bu_H-\bu_h,\bu_h,\e_t) \\
&= \frac{d}{dt}\big\{b(\bu_H,\bu_H-\bu_h,\e)+b(\bu_H-\bu_h,\bu_h,\e)\big\}
-b(\bu_{H,t},\bu_H-\bu_h,\e) \\
&-b(\bu_H,\bu_{H,t}-\bu_{h,t},\e)-b(\bu_{H,t}-\bu_{h,t},\bu_h,\e)
-b(\bu_H-\bu_h,\bu_{h,t},\e).
\end{align*}
And hence
\begin{align}\label{eeu05}
\sigma_1(t)\Lambda_h(\e_t) &= \sigma_1(t)\big( b(\bu_H,\bu_H-\bu_h,\e_t) +b(\bu_H-\bu_h,\bu_h,\e_t)\big) \nonumber \\
&= \frac{d}{dt}\big\{\sigma_1(t)\big(b(\bu_H,\bu_H-\bu_h,\e)+b(\bu_H-\bu_h,\bu_h,\e)
\big)\big\} \nonumber \\
&-\sigma_{1,t}(t)\big(b(\bu_H,\bu_H-\bu_h,\e)+b(\bu_H-\bu_h,\bu_h,\e)
\big)-\sigma_1(t)b(\bu_{H,t},\bu_H-\bu_h,\e) \nonumber \\
&-\sigma_1(t)\big(b(\bu_H,\bu_{H,t}-\bu_{h,t},\e)+b(\bu_{H,t}-\bu_{h,t},\bu_h,\e)
+b(\bu_H-\bu_h,\bu_{h,t},\e)\big).
\end{align}
As seen in (\ref{inter02}), we have
\begin{align*}
&-b(\bu_{H,t},\bu_H-\bu_h,\e)-b(\bu_H-\bu_h,\bu_{h,t},\e) \\
\le & C\|\bu_H-\bu_h\|\|\e\|_1\big\{\|\bu_{H,t}\|_1^{1/2}\|\bu_{H,t}\|_2^{1/2}
+\|\bu_{h,t}\|_1^{1/2}\|\bu_{h,t}\|_2^{1/2}\big\}
\end{align*}
Use Lemmas \ref{est1.uh} and \ref{est1.uH} to conclude that
\begin{equation}\label{eeu06}
-b(\bu_{H,t},\bu_H-\bu_h,\e)-b(\bu_H-\bu_h,\bu_{h,t},\e) \le K(\tau^*(t))^{-3/2} \|\bu_H-\bu_h\|\|\e\|_1.
\end{equation}
Similarly
\begin{equation}\label{eeu07}
-b(\bu_H,\bu_H-\bu_h,\e)-b(\bu_H-\bu_h,\bu_h,\e) \le K(\tau^*(t))^{-1/2}
\|\bu_H-\bu_h\|\|\e\|_1
\end{equation}
and
\begin{equation}\label{eeu08}
-b(\bu_H,\bu_{H,t}-\bu_{h,t},\e)-b(\bu_{H,t}-\bu_{h,t},\bu_h,\e) \le K(\tau^*(t))^{-1/2}
\|\bu_{H,t}-\bu_{h,t}\|\|\e\|_1.
\end{equation}
Incorporate (\ref{eeu06})-(\ref{eeu08}) in (\ref{eeu05}) and integrate with respect to time. Re-use (\ref{eeu07}) to find
\begin{align}\label{eeu09}
\int_0^t \sigma_1(s)\Lambda_h(\e_s)~ds \le & \sigma_1(t)\big(b(\bu_H,\bu_H-\bu_h,\e) +b(\bu_H-\bu_h,\bu_h,\e)\big) \nonumber \\
+K\int_0^t &~e^{2\alpha s} \big\{(\tau^*(s))^{1/2} \|\bu_H-\bu_h\|+(\tau^*(s))^{3/2}
\|\bu_{H,t}-\bu_{h,t}\|\big\}\|\e(s)\|_1 ds \nonumber \\
\le & Ke^{2\alpha t}(\tau^*(t))^{3/2}\|\bu_H-\bu_h\|\|\e\|_1+C\int_0^t \sigma(s)\|\e(s)\|_1^2 ds \nonumber \\
& +K\int_0^t e^{2\alpha s} \big\{\|\bu_H-\bu_h\|^2+(\tau^*(s))^2\|\bu_{H,t}-\bu_{h,t}\|^2\big\}
~ds \nonumber \\
\le & \frac{\mu}{6}\sigma_1(t)\|\e\|_1^2+K\sigma(t)\|\bu_H-\bu_h\|^2+C\int_0^t \sigma(s) \|\e(s)\|_1^2 ds \nonumber \\
& +K\int_0^t e^{2\alpha s} \big\{\|\bu_H-\bu_h\|^2+(\tau^*(s))^2\|\bu_{H,t}-\bu_{h,t}\|^2\big\}
~ds.
\end{align}
Put (\ref{eeu09}) in (\ref{eeu04}) to obtain
\begin{align}\label{eeu10}
2\int_0^t \sigma_1(s)\|\e_s(s)\|^2 ds +\frac{\mu}{3}\sigma_1(t)\|\e\|_1^2 \le
C\int_0^t e^{2\alpha s}\|\te(s)\|_1^2 ds+C\int_0^t \sigma(s)\|\e(s)\|_1^2 ds \\
+C\sigma_1(t)\|\te\|_1^2+K\sigma(t)\|\bu_H-\bu_h\|^2+K\int_0^t e^{2\alpha s} \big\{\|\bu_H-\bu_h\|^2+(\tau^*(s))^2\|\bu_{H,t}-\bu_{h,t}\|^2\big\}~ds. \nonumber
\end{align}
Now, use Lemmas \ref{inter}, \ref{pee}, \ref{inter10} and \ref{err.et}. For the remaining part, we use (\ref{est.hH}) and Lemma $5.4$ and Theorem $5.1$ from \cite{GP11}. This completes the rest of the proof.
\end{proof}
\noindent
We present below, the pressure error estimate.
\begin{lemma}\label{prs} 
Under the assumptions of Lemma \ref{neg} and additionally that the assumption $({\bf B2'})$
holds, we have
\begin{equation}\label{prs1}
\|(p_h-p^h)(t)\| \le K(t)(\tau^*(t))^{-1}H^2.
\end{equation}
\end{lemma}

\begin{proof}
The LBB condition $({\bf B2'})$ tells us that, for $t>0$
\begin{equation}\label{prs01}
\|p_h-p^h\|_{L^2/\R} \le C\sup_{0\neq \bphi_h\in\bH_h} \frac{(p_h-p^h, \nabla\cdot\bphi_h)}{\|\bphi_h\|_1}.
\end{equation}
From (\ref{dwfh}) and (\ref{2lvlH2}), we have, for $\bphi_h\in \bH_h$
\begin{align}\label{prs02}
(p_h-p^h, \nabla \cdot \bphi_h)=(\e_t,\bphi_h) +\mu a(\e,\bphi_h)+\int_0^t \beta(t-s) a(\e(s),\bphi_h)~ds-\Lambda_h(\bphi_h).
\end{align}
Using (\ref{pee01a}) and (\ref{inter02}), we obtain from (\ref{prs02})
\begin{align}\label{prs03}
(p_h-p^h, \nabla \cdot \bphi_h) \le C\|\bphi_h\|_1\big\{\|\e_t\|_{-1,h}+\|\e\|_1+\|\te\|_1+
\big(e^{-2\alpha t}\int_0^t e^{2\alpha s}\|\te(s)\|_1^2 ds\big)^{1/2}\big\} \nonumber \\
+K(\tau^*(t))^{-1/4}\|\bu_H-\bu_h\|\|\bphi_h\|_1,
\end{align}
where
\begin{align*}
\|\e_t\|_{-1,h} =\sup_{0\neq \bphi_h\in\bH_h}\frac{<\e_t,\bphi_h>}{\|\bphi_h\|_1}.
\end{align*}
Taking supremum over a bigger set, we find that
\begin{align}\label{neg.h}
\|\e_t\|_{-1,h} \le \|\e_t\|_{-1} =\sup_{0\neq\bphi\in\bH_0^1} \frac{<\e_t,\bphi>}{\|\bphi\|_1}.
\end{align}
Incorporate this in (\ref{prs03}) and now, for $0\neq\bphi_h$, divide by $\|\bphi_h\|_1$.
Use (\ref{prs01}) and then Lemmas \ref{eeu}, \ref{inter10} and \ref{inter} to get
\begin{align}\label{prs04}
\|p_h-p^h\|_{L^2/\R} \le C\|\e_t\|_{-1}+K(t)(\tau^*(t))^{-1}H^2.
\end{align}
As usual, we have rewritten $\bu_H-\bu_h$ as $(\bu-\bu_h)-(\bu-\bu_H)$ and used Theorem \ref{errest}. \\
For the negative norm, we have, for $\bphi\in\bH_0^1$
$$ (\e_t,\bphi)= (\e_t,\bphi-P_h\bphi)+(\e_t,P_h\bphi). $$
Since $P_h\bphi\in\bJ_h$, we use the error equation (\ref{err})
\begin{align*}
(\e_t,\bphi)= (\e_t,\bphi-P_h\bphi)-\mu a(\e,P_h\bphi)-\int_0^t \beta(t-s) a(\e(s), P_h\bphi)~ds+\Lambda_h(P_h\bphi),
\end{align*}
where $\Lambda_h$ is given by (\ref{lamb}).
Using discrete incompressibility condition, $H^1$-stability of $P_h$, $({\bf B1})$ and (\ref{inter02}), we obtain
\begin{align}\label{prs05}
(\e_t,\bphi) \le &~Ch\|\e_t\|\|\bphi\|_1+C\big(\|\e\|_1+\int_0^t \beta(t-s)\|\e(s)\|_1 ds
\big)\|\bphi\|_1 \nonumber \\
+&~C\|\bu_H-\bu_h\|\big(\|\bu_H\|_1^{1/2}\|\bu_H\|_2^{1/2}+\|\bu_H\|_1^{1/2}
\|\bu_H\|_2^{1/2}\big)\|\bphi\|_1
\end{align}
We can rewrite the integral term as in (\ref{prs03}) and use the estimates used in (\ref{prs03}) to obtain from (\ref{prs05})
\begin{align*}
\|\e_t\|_{-1} \le Ch\|\e_t\|+K(t)(\tau^*(t))^{-1}H^2.
\end{align*}
Incorporate this in (\ref{prs04}).
\begin{align}\label{prs06}
\|p_h-p^h\|_{L^2/\R} \le CH\|\e_t\|+K(t)(\tau^*(t))^{-1}H^2.
\end{align}
Assuming the following sub-optimal estimate (since it is of $O(H)$)
\begin{align}\label{sub-op}
\|\e_t\| \le K(t)(\tau^*(t))^{-1}H,
\end{align}
we complete the rest of the proof.
\end{proof}
\begin{remark}
It is sufficient to have a sub-optimal estimate of $\|\e_t\|$ to attain our result. Instead, had we chosen to estimate the first term on the right-hand side of (\ref{prs02}) as
$$ (\e_t,\bphi_h) \le \|\e_t\|\|\bphi_h\| \le C\|\e_t\|\|\bphi_h\|_1, $$
we would require an optimal estimate of $\|\e_t\|$, which would be $O((\tau^*(t))^{-3/2}H^2)$.
And then the pressure error estimate would read $O((\tau^*(t))^{-3/2}H^2)$, which is worse that the estimate we have attained here, that is (\ref{prs1}).
\end{remark}

\begin{lemma}\label{sub-op0}
Under the assumptions of Lemma \ref{neg}, the following holds:
\begin{align}\label{sub-op1}
\|\e_t\| \le K(t)(\tau^*(t))^{-1}H.
\end{align}
\end{lemma}

\begin{proof}
We differentiate the error equation (\ref{err}) with respect to time and then incorporate (\ref{pee01a}).
\begin{align}\label{errt}
(\e_{tt},\bphi_h) +\mu a(\e_t,\bphi_h)+\gamma a(\e,\bphi_h)-\delta\gamma a(\te,\bphi_h) +\delta^2\int_0^t \beta(t-s) a(\te(s), \bphi_h)~ds \nonumber \\
= \Lambda_{h,t}(\bphi_h)~~\forall \bphi_h \in \bJ_h,
\end{align}
where
\begin{align}\label{lambt}
\Lambda_{h,t}(\bphi_h)=& b(\bu_{H,t},\bu_H-\bu_h,\bphi_h)+b(\bu_H,\bu_{H,t}-\bu_{h,t},
\bphi_h)+b(\bu_{H,t}-\bu_{h,t},\bu_h,\bphi_h) \nonumber \\
&+b(\bu_H-\bu_h,\bu_{h,t},\bphi_h).
\end{align}
Choose $\bphi_h=\sigma_1(t)\e_t$ in (\ref{errt}) to find
\begin{align}\label{sub-op01}
\frac{1}{2}\frac{d}{dt}\big\{\sigma_1(t)\|\e_t\|^2\big\}+\mu\sigma_1(t) \|\e_t\|_1^2 \le 
\frac{\sigma_{1,t}(t)}{2}\|\e_t\|^2+\sigma_1(t)\Lambda_{h,t}(\e_t) \nonumber \\
C\sigma_1(t)\|\e_t\|_1\big\{\|\e\|_1+\|\te\|_1+\int_0^t \beta(t-s)\|\te(s)\|_1 ds \big\}. 
\end{align}
Using Lemma \ref{nonlin}, we observe that
\begin{align}\label{sub-op02}
b(\bu_{H,t},\bu_H-\bu_h,\sigma_1(t)\e_t)&= \frac{1}{2}\sigma_1(t)\big\{((\bu_{H,t}\cdot\nabla)
(\bu_H-\bu_h), \e_t)-((\bu_{H,t}\cdot\nabla)\e_t, \bu_H-\bu_h)\big\}  \nonumber \\
\le C\sigma_1(t)\|\bu_{H,t}\|^{1/2}\|\bu_{H,t}\|_1^{1/2}&\big(\|\bu_H-\bu_h\|_1\|\e_t\|^{1/2} \|\e_t\|_1^{1/2}+\|\e_t\|_1\|\bu_H-\bu_h\|^{1/2}\|\bu_H-bu_h\|_1^{1/2}\big) \nonumber \\
\le C\sigma_1(t)&\|\bu_{H,t}\|_1\|\bu_H-\bu_h\|_1\|\e_t\|_1.
\end{align}
Repeating (\ref{sub-op02}), we obtain
\begin{align}\label{sub-op03}
\sigma_1(t)\Lambda_{h,t} \le &~C\sigma_1(t)\big(\|\bu_{H,t}\|_1+\|\bu_{h,t}\|_1\big)
\|\bu_H-\bu_h\|_1\|\e_t\|_1 \nonumber \\
&~C\sigma_1(t)\big(\|\bu_H\|_1+\|\bu_h\|_1\big)\|\bu_{H,t}-\bu_{h,t}\|_1\|\e_t\|_1
\end{align}
Incorporate (\ref{sub-op03}) in (\ref{sub-op01}). Use Young's inequality and kickback argument. Then integrate the resulting inequality to observe that
\begin{align}\label{sub-op04}
\sigma_1(t)\|\e_t\|^2+\mu\int_0^t \sigma_1(s) \|\e_s(s)\|_1^2 ds &\le C\int_0^t \sigma(s) \|\e_s(s)\|^2~ds+C\int_0^t \sigma_1(s)\big\{\|\e(s)\|_1^2+\|\te(s)\|_1^2\big\}~ds \nonumber \\
+C&\int_0^t \e^{2\alpha s}\|\te(s)\|_1^2 ds+K(t)H^2\int_0^t \sigma(s)(\|\bu_{H,s}(s)\|_1^2 +\|\bu_{h,s}(s)\|_1^2)~ds \nonumber \\
+K&\int_0^t \sigma_1(s) \|(\bu_{H,s}-\bu_{h,s})(s)\|_1^2 ds.
\end{align}
We have used Lemmas \ref{est.uh} and \ref{est.uH} to bound $\|\bu_H\|_1$ and $\|\bu_h\|_1$.
As earlier, we have rewritten $\bu_H-\bu_h$ as $(\bu-\bu_h)-(\bu-\bu_H)$ and have used Theorem \ref{errest}.
Use Lemmas \ref{est1.uh} and \ref{est1.uH} to bound $\|\bu_{H,t}\|_1$ and $\|\bu_{h,t}\|_1$ under integral on the right-hand side of (\ref{sub-op04}). Next, use Lemmas \ref{pee} and \ref{inter}. Use Lemma \ref{err.so} to estimate the last term on the right-hand side of (\ref{sub-op04}).
\begin{align}\label{sub-op05}
\sigma_1(t)\|\e_t\|^2+\int_0^t \sigma_1(s) \|\e_s(s)\|_1^2 ds \le C\int_0^t \sigma(s)
\|\e_s(s)\|^2~ds+K(t)H^2.
\end{align}
To estimate the term on the right-hand side of (\ref{sub-op05}), we put $\bphi_h=\sigma(t) \e_t$ in (\ref{err}).
\begin{align}\label{sub-op06}
\sigma(t)\|\e_t\|^2+\frac{\mu}{2}\frac{d}{dt}\{\sigma(t)\|\e\|_1^2\}= \frac{1}{2}\sigma_t(t)
\|\e\|_1^2-\sigma(t)\int_0^t \beta(t-s) a(\e(s),\e_t)~ds+\sigma(t)\Lambda_h(\e_t).
\end{align}
We handle the integral term as in (\ref{int}).
We take care of the non-linear term as earlier.
$$ \Lambda_h(\e_t) \le K\|\e_t\|\|\bu_H-\bu_h\|_1(\|\bu_h\|_2^{1/2}+\|\bu_H\|_2^{1/2}). $$
Incorporate this in (\ref{sub-op06}) and after simplifying, integrate with respect to time.
\begin{align}\label{sub-op07}
\sigma(t)\|\e\|_1^2+\int_0^t \sigma(s)\|\e_s(s)\|^2 ds \le C\int_0^t e^{2\alpha s}
\|\e(s)\|_1^2 ds+K\int_0^t e^{2\alpha s}\|\E(s)\|_1^2 ds.
\end{align}
Apply Lemma \ref{err.so} to obtain
\begin{align}\label{sub-op08}
\sigma(t)\|\e\|_1^2+\int_0^t \sigma(s)\|\e_s(s)\|^2 ds \le C\int_0^t e^{2\alpha s}
\|\e(s)\|_1^2 ds+K(t)H^2.
\end{align}
For the term on the right-hand side of (\ref{sub-op08}), we put $\bphi_h=e^{2\alpha t}\e$
in (\ref{err}) to find
\begin{align}\label{sub-op09}
\frac{1}{2}\frac{d}{dt}\{e^{2\alpha t}\|\e\|^2\}+\mu e^{2\alpha t}\|\e\|_1^2+ \int_0^t
\beta(t-s) a(\e(s), e^{2\alpha t}\e)~ds= \alpha e^{2\alpha t}\|\e\|^2+e^{2\alpha t} \Lambda_h(\e).
\end{align}
Note that
$$ \Lambda_h(\e) \le (\|\bu_h\|_1+\|\bu_H\|_1)\|\e\|_1\|\bu_H-\bu_h\|_1 \le K\|\e\|_1
\|\bu_H-\bu_h\|_1. $$
We incorporate this in (\ref{sub-op09}) and then integrate with respect to time. Drop the double integral term.
\begin{align*}
e^{2\alpha t}\|\e\|^2+\int_0^t e^{2\alpha s}\|\e(s)\|_1^2 ds \le C\int_0^t e^{2\alpha s}
\|\e(s)\|^2 ds+K\int_0^t e^{2\alpha s}\|(\bu_H-\bu_h)(s)\|_1^2 ds.
\end{align*}
Use Lemmas to observe that
$$ \|\e\|^2+e^{-2\alpha t}\int_0^t e^{2\alpha s}\|\e(s)\|_1^2 ds \le K(t)H^2. $$
Use this to complete the estimate in (\ref{sub-op08}), which, we use, in turn, in (\ref{sub-op05}) to complete the rest of the proof.
\end{proof}

\begin{remark}
In the Lemmas (\ref{err.et}) and (\ref{prs}), we have obtained
\begin{align}\label{rmk0}
\|(\bu_h-\bu^h)(t)\|_1+\|(p_h-p^h)(t)\| \le K(t)(\tau^*(t))^{-1}H^2.
\end{align}
Combining this with Theorem \ref{errest}, we have
\begin{align}\label{rmk}
\|(\bu-\bu^h)(t)\|_1+\|(p-p^h)(t)\| \le K(t)\big\{(\tau^*(t))^{-1/2}h+(\tau^*(t))^{-1} H^2\big\}.
\end{align}
We note that for $H=O(h^{1/2})$, both Galerkin FE and two-level methods produce similar results, away from initial time. But while the first method solves the full non-linear equation on a grid of mesh-size $h$, the second only does that on a coarse mesh (grid of mesh-size $H,~H<<h$) and supplements it by solving a linearized problem on the fine grid (mesh-size $h$).
\end{remark}
\begin{remark}
Note that the improvement is shown only in energy norm. In other words, in the $L^2$-norm of velocity error, we see no improvement of the two-level over Galerkin FE. The result achieved here, i.e, (\ref{rmk0}), is nothing but super-convergence result and hence, we can not expect improved estimate for $L^2$ velocity error (at least not by the proof techniques used here). We would like to point that out that this super-convergence is also established in \cite{GP08} for Navier-Stokes equations, but for non-linear Galerkin method, which is similar in nature to two-level or two-grid methods. And in \cite{FGN12a}, a static two-grid method (a discrete steady Oseen-type problem is considered at the second level ) is discussed for Navier-Stokes equations, where numerical results show that, for mini element, $L^2$ velocity error has same rate of convergence for both Galerkin FE and the static two-grid methods.
\end{remark}

\end{document}